\documentclass[11pt]{article}

\usepackage{graphicx,psfrag, subcaption}
\usepackage{enumitem}

\usepackage{amssymb,amsmath,amsthm}
\usepackage{fullpage}

\usepackage{amsmath,amsthm,amsfonts,amssymb,epsfig,verbatim, fge}
\usepackage[normalem]{ulem}
\usepackage{xcolor}
\usepackage{circledsteps}
\usepackage{url}
\usepackage{breqn}
\usepackage{bbm}
\usepackage[colorlinks, citecolor=blue]{hyperref}
\newtheorem{problem}{Problem}
\newtheorem{mainthm}{Theorem}
\newtheorem{thm}{Theorem}[section]
\newtheorem{lem}[thm]{Lemma}
\newtheorem{prop}[thm]{Proposition}

\newtheorem{claim}[thm]{Claim}
\newtheorem{rem}{Remark}

\numberwithin{equation}{section}

\newcommand{\E}{\mathbb{E}}

\newcommand{\prob}{\mathbb{P}}

\newcommand{\R}{\mathbb{R}}
\newcommand{\C}{\mathbb{C}}
\newcommand{\cS}{\mathcal{S}}

\newcommand{\ccW}{\overline W}

\newcommand{\cB}{\mathcal{B}}
\newcommand{\cE}{\mathcal{E}}

\newcommand{\eps}{\varepsilon}
\renewcommand{\setminus}{\mathbin{\fgebackslash}}

\title{Limiting Root Distribution of Random Log-concave Polynomials}
\author{Ohad Noy Feldheim \thanks{Email: ohad.feldheim@mail.huji.ac.il. Partly supported by ISF personal research grant 3541/24.} \and Arnab Sen \thanks{Email: arnab@umn.edu. Partly supported by Simons Foundation MP-TSM-00002716.}}
\date{\today}
\begin{document}
\maketitle
%\section{}
%\subsection{}

\begin{abstract}
 
Log-concave polynomials arise naturally in combinatorics, geometry,
and probability. We introduce a natural family of random
log-concave polynomials and study their limiting root distributions.
Our starting point is the \emph{uniform model}, in which the
coefficient vector $(a_0,\ldots,a_n)$ is chosen uniformly from
$[0,1]^{n+1}$ subject to the log-concavity constraint. More
generally, for $\alpha\ge0$, we define
\[
P_{n,\alpha}(z)=\sum_{i=0}^n a_i^{n^\alpha}z^i.
\]
Thus, $\alpha=0$ corresponds to the uniform model. For each fixed
$0\le\alpha<1$, we prove that the empirical root measure converges
almost surely to the uniform probability measure on the unit circle.
For the critical case $\alpha=1$, which we call the {\emph beta model}, the
empirical root measure instead converges in probability to a
rotationally invariant probability measure that is absolutely
continuous with respect to planar Lebesgue measure.
\end{abstract}

\section{Introduction}

A polynomial $P_n(z)=\sum_{i=0}^n a_i z^i$ with non-negative real coefficients is called \emph{log-concave} if its coefficients form a log-concave sequence, that is,
\[
a_{i-1}a_{i+1}\le a_i^2, \qquad 1\le i\le n-1.
\]
Log-concavity arises naturally and extensively in algebraic and enumerative combinatorics, as well as in geometry and probability, see the surveys \cite{stanley1989log, brenti1994log, branden2015unimodality}. Many generating functions that enumerate combinatorial structures have log-concave coefficients, and hence give rise to log-concave polynomials. Examples include the matching generating polynomials of graphs \cite{heilmann1972theory}, the polynomials obtained from the absolute values of the coefficients of chromatic polynomials of graphs and characteristic polynomials of matroids \cite{huh2012milnor, huhkatz2012bergman, adiprasito2018hodge}, as well as Steiner polynomials of convex bodies \cite{schneider2013convex, henk2012steiner}, among many others.

The zeros of generating functions play a central role in coefficient asymptotics. In statistical mechanics, the zeros of the partition function likewise carry important information about phase transitions through the Lee-Yang theory. Moreover, the location of zeros can have strong probabilistic consequences: certain zero-free regions imply central limit and local limit theorems for the normalized coefficients, see \cite{lebowitz2016central, michelen2019roots, michelen2026geometry}.

By Newton's inequalities, if a polynomial with non-negative coefficients has only real zeros, then its coefficient sequence is (ultra) log-concave. The converse fails in general, as shown for instance by the $q$-factorial polynomials $[n]_z!:=\prod_{j=1}^n (1+z+\cdots+z^{j-1})$, see \cite[Example 1.1]{branden2015unimodality}. In fact, the zeros of the $q$-factorial polynomials are roots of unity, and their normalized zero counting measures converge to the uniform probability measure on $\mathbb S^1$. The truncated exponential polynomials $s_n(z)=\sum_{k=0}^n z^k/k!$ provide another family of log-concave polynomials whose rescaled zeros converge to a probability measure supported on the Szeg\H{o} curve $\{z\in\C: |ze^{1-z}|=1,\ |z|\le 1\}$ \cite{pritsker1997szego}, which, in particular,  is not concentrated on the unit circle.  Although log-concavity is strictly weaker than real-rootedness, a sufficiently strengthened form of log-concavity does imply
real-rootedness. More precisely, Kurtz \cite{kurtz1992sufficient} proved that if a polynomial
$\sum_{i=0}^n a_i z^i$ has positive coefficients satisfying $a_i^2>4a_{i-1}a_{i+1}$ for every $1\le i\le n-1$, then all its zeros are real and negative; see also
\cite{handelman2013arguments} for refinements.

Given the prominence of log-concave polynomials, it is natural to study random models of log-concave polynomials, in order to better understand the typical effect of this constraint on 
the limiting distribution of their roots as the degree tends to infinity.
To this end, we begin by studying the root distribution of a uniformly chosen log-concave polynomial of degree $n$, which we call the {\em uniform model}. We show that this model's coefficient profile is rather flat in the bulk, so its root distribution still converges to the unit-circle attractor. We then obtain a second natural model, by raising each coefficient of the uniform model to the power  $n$, which call the {\em beta model}. We show that this model has a nontrivial absolutely continuous, radially symmetric limiting root distribution, which we compute explicitly.

\subsection{The Uniform Model}
Consider a coefficient vector $(a_0, a_1, \ldots, a_n)$ chosen uniformly from $[0,1]^{n+1}$ subject to the log-concavity constraint. More precisely, let $U_0, U_1, \ldots, U_n$ be i.i.d.\ $U(0,1)$ random variables, define the event $\mathcal{L}:=\{U_{i-1}U_{i+1}\le U_i^2:\ 1\le i\le n-1\}$, namely the sequence $U_0, U_1, \ldots, U_n$ is log-concave. Consider
\begin{equation}\label{def:a}
    (a_0, a_1, \ldots, a_n)\stackrel{d}{=}(U_0, U_1, \ldots, U_n)\mid \mathcal{L},
\end{equation}
which defines the uniform distribution on the set of log-concave sequences in $[0,1]^{n+1}$. 
We call $P_n(z) = \sum_{i=0}^n a_i z^i$ a random log-concave polynomial drawn from the \emph{uniform model}.

Since log-concavity and the zero set are invariant under positive scalar multiplication of the coefficients, replacing $[0,1]^{n+1}$ by $[0,T_n]^{n+1}$ for any $T_n>0$ yields an equivalent model, and hence we may assume $T_n=1$.

Another equivalent description of the uniform model is obtained by introducing a random convex sequence.  Given a vector $x=(x_0,x_1,\ldots,x_n)$, define the first differences $\Delta_i x=x_i-x_{i-1}$ for $1\le i\le n$ and the second differences $\Delta_i^2 x=\Delta_i x-\Delta_{i-1}x$ for $2\le i\le n$. Let $X_0,X_1,\ldots,X_n$ be i.i.d.\ $\mathrm{exp}(1)$ random variables, and write $X=(X_0,X_1,\ldots,X_n)$. We say that $(\xi_0,\xi_1,\ldots,\xi_n)$ is a \emph{random convex sequence} if
\begin{equation}\label{def:random_convex}
(\xi_0,\xi_1,\ldots,\xi_n)\stackrel{d}{=}(X_0,X_1,\ldots,X_n)\mid \mathcal{C},
\end{equation}
where $\mathcal{C}:=\{\Delta_i^2 X\ge 0:\ 2\le i\le n\}$
is the convexity event. Since $U_i:=e^{-X_i}$ are i.i.d.\ $U(0,1)$ random variables, and the event $(U_0,U_1,\ldots,U_n)\in\mathcal L$ is equivalent to $(X_0,X_1,\ldots,X_n)\in\mathcal C$, it follows that
\begin{equation}\label{eq:xi_a}
(a_0,a_1,\ldots,a_n)\stackrel{d}{=}(e^{-\xi_0},e^{-\xi_1},\ldots,e^{-\xi_n}),
\end{equation}
hence the corresponding polynomial $P_n(z) = \sum_{i=0}^n e^{-\xi_i} z^i $ is drawn from the uniform model.

As it turns out, these random polynomials have the same limiting root distribution as the unconditioned polynomial with i.i.d.\ coefficients uniformly distributed on $[0,1]$ (See Figure~\ref{fig:root of unif}). The latter is an example of a Kac polynomial, and its limiting root distribution is the uniform probability measure on the unit circle; see \cite{shepp1995complex,ibragimov1997roots}. More precisely,
 \begin{mainthm}\label{thm:uniform}
    Write $P_n(z) = \sum_{i=0}^n a_i z^i$ for a random log-concave polynomial of degree $n$ drawn from the uniform model, and denote its roots by $\zeta_{n,1},\dots, \zeta_{n,n}$. Then, as $n \to \infty,$ almost surely, 
 \[ d_L\Big (n^{-1} \sum_{k=1}^n \delta_{\zeta_{n,k}}, \mathsf{U}(\mathbb{S}^1) \Big) \to  0, \]   
    where $\mathsf{U}(\mathbb{S}^1)$ is the uniform distribution on the unit circle $\mathbb{S}^1$ and $d_L$ is  the L\'evy-Prokhorov metric on $\C$.
\end{mainthm}
\begin{figure}[htbp]
  \centering
\includegraphics[scale=.9]{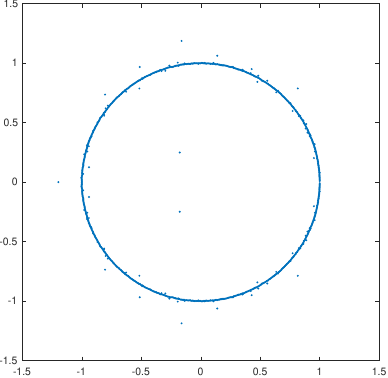}
 \caption{Roots of the uniform model for $ n=1200$.}
 \label{fig:root of unif}
\end{figure}

Theorem~\ref{thm:uniform} shows that the macroscopic root distribution
of the uniform model belongs to the same universality class as that of
a Kac polynomial. The underlying reason is that conditioning on
log-concavity causes the coefficients to vary only polynomially in
$n$; in fact, most coefficients remain of comparable order. This
places the model within the scope of classical Erd\H{o}s--Tur\'an-type
results, such as those in \cite{hughes2008zeros}; see
Theorem~\ref{thm:uniform_dist_root_criteria} below. At the opposite extreme, sufficiently strong log-concavity can force real-rootedness, as in Kurtz's criterion above. To interpolate between these two regimes, for $\alpha\in [0,\infty)$ we introduce the
$\alpha$-power model
\begin{equation}\label{eq:b=a^n}
P_{n, \alpha}(z):=\sum_{i=0}^n b_i z^i, \qquad \text{where } \ (b_0,\dots,b_n)\stackrel{d}{=}(a_0^{n^\alpha},\dots,a_n^{n^\alpha}), 
\end{equation}
where $(a_0,\dots,a_n)$ is sampled from the uniform model defined in
\eqref{def:a}. The case $\alpha=0$ corresponds to the uniform model.  For $\alpha>0$ raising all coefficients to the same power
$n^\alpha\ge1$ preserves log-concavity and strengthens its local
effect, this family provides a natural interpolation between the
uniform model and more strongly log-concave models.

The natural unconditioned counterpart is
\begin{equation}\label{def:unconditioned_alpha_model}
  Q_{n, \alpha}(z):=\sum_{i=0}^n U_i^{n^\alpha} z^i,  
\end{equation}
where $U_0,U_1,\dots,U_n$ are i.i.d.\ $U(0,1)$ random variables.

For $0 \le \alpha<1$, the macroscopic root distributions of both models
remain in the universality class of Kac polynomials; see  Proposition~\ref{prop:alpha_less_than_one_convergence} and Proposition~\ref{prop:unconditioned_alpha_less_than_one_convergence}. When $\alpha>1$, the polynomial $P_{n,\alpha}$ becomes ill-scaled and its empirical root measure 
 degenerates to $(1/2)\delta_0+ (1/2)\delta_\infty$  on 
$\widehat{\C}=\C\cup\{\infty\}$; see Proposition~\ref{prop:alpha_greater_than_one_convergence}. This leaves $\alpha=1$ as the critical case. At this scale,
$Q_{n,1}$ still retains the uniform limiting root distribution on the
unit circle; see Proposition~\ref{prop:unconditioned_beta}. By
contrast, we show that $P_{n,1}$ has a nontrivial two-dimensional limiting root
distribution, revealing the effect of log-concavity. We call
$P_{n,1}$ the \emph{beta model} and study it in the next subsection.

\subsection{The Beta Model}\label{sec:beta} 
Since $U^n\sim\operatorname{Beta}(1/n,1)$ when $U\sim U(0,1)$, the
coefficient vector $(b_0,b_1,\ldots,b_n)$ of $P_{n,1}$ can
equivalently be obtained by conditioning a vector of i.i.d.\
$\operatorname{Beta}(1/n,1)$ random variables to be log-concave. This
justifies the name of the model. Using the random convex-sequence
representation in \eqref{def:random_convex}, the beta model can also
be written as
\[ P_{n,1}(z)=\sum_{i=0}^n e^{-n\xi_i}z^i.
\]
 Simulations of this model are provided in Figure~\ref{fig:side-by-side}.

 The next theorem, which is the main result of the paper, identifies the limiting root distribution of random polynomials drawn from the beta model. As in the uniform model, the limiting measure is rotationally symmetric; however, in the beta model it is absolutely continuous with respect to two-dimensional Lebesgue measure on $\mathbb C$
 
\begin{mainthm}\label{thm:main_beta}
    Let $P_n(z) = \sum_{i=0}^n b_i z^i$ be a random log-concave polynomial of degree $n$ drawn from the beta model, and denote its roots by $\zeta_{n,1},\dots, \zeta_{n,n}$. Then, as $n\to\infty$, the empirical root measure converges in probability to $\mu$ in L\'evy-Prokhorov distance, that is,
\begin{equation}\label{eq:main_1_2} d_L\Big (n^{-1} \sum_{k=1}^n \delta_{\zeta_{n,k}}, \mu \Big) \stackrel{p}{\to} 0, \end{equation}
    where $\mu$ is the absolutely continuous probability measure on $\mathbb{C}$ with density
    \[
\frac{d\mu}{d z} 
= \frac{1}{\pi |z|^2 \bigl(4 + |\log |z||\bigr)^2},
\qquad z \in \mathbb{C}\setminus\{0\}.
\]
\end{mainthm}

\begin{rem}
The limiting density of $\mu$ is radially symmetric, and the radial component has the density
\[ f(r) = \frac{2}{r(4+ |\log r|)^2}, \ \ r >0.\]
\end{rem}

\begin{figure}[t]
  \centering
  \begin{subfigure}{0.48\linewidth}
    \centering
      \includegraphics[width=\linewidth]{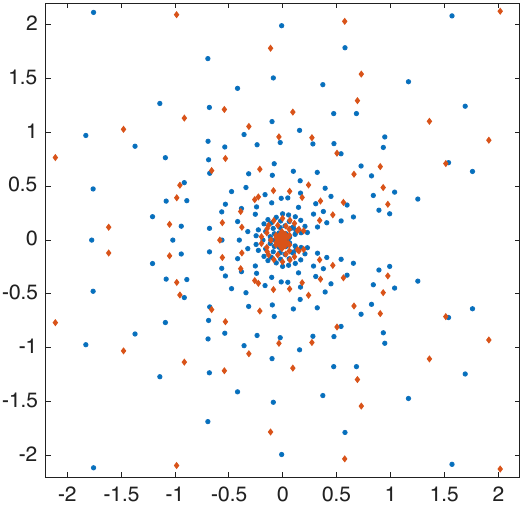}
    \label{fig:left}
  \end{subfigure}\hfill
  \begin{subfigure}{0.48\linewidth}
    \centering
    \includegraphics[width=\linewidth]{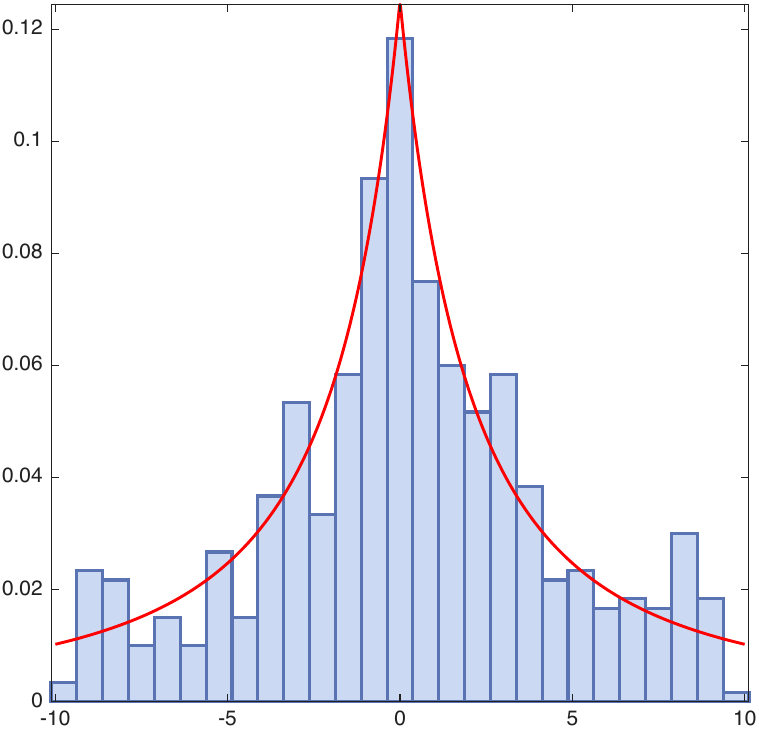}
    \label{fig:right}
  \end{subfigure}
 \caption{Left: roots of the beta model for $n=400$ (orange) and $n=800$ (blue), restricted to $[-2,2]^2$. Right: histogram of $\log|\zeta_{n, k}|$ for $n=800$, overlaid with the theoretical log-radial density $f_{\log |\zeta|}(x)=\frac{2}{(4+|x|)^2}$ for $x\in\mathbb{R}$. \\[0.6ex]
Although the limiting root distribution is rotationally symmetric, the visible cone around the positive $x$-axis that seems to contain no zeros is somewhat puzzling. Since all coefficients are strictly positive, the polynomial has no positive real root, so the apparent zero-free cone is likely a finite-$n$ effect. Indeed, as $n$ increases from $400$ to $800$, the cone already becomes noticeably narrower. Because the coefficients vary exponentially in $n$, we compute the roots using variable-precision arithmetic in MATLAB, which makes simulations for larger $n$ increasingly expensive.}
\label{fig:side-by-side}
\end{figure}

\subsection{Discussion}

The starting point of our analysis is an explicit description of the
joint distribution of the random log-concave sequence
$(a_0,\ldots,a_n)$, or equivalently of the random convex sequence
$(\xi_0,\ldots,\xi_n)$ defined by $\xi_i=-\log a_i$.
Lemma~\ref{lem:rep} gives a representation as a mixture over the
index $R$ of the largest coefficient. Conditional on $R$, the
$\xi_i$ are linear combinations of i.i.d.\ mean-one exponential
random variables. Theorem~\ref{thm:uniform} then follows by applying
the equidistribution criterion of \cite{hughes2008zeros}. 

For the beta model, we use concentration of the weighted sums of
independent exponential random variables in Lemma~\ref{lem:rep}
to approximate the logarithms of the coefficients by a deterministic
convex profile $\theta:(0,1)\to[0,\infty)$. More precisely,
\[
b_i=\exp\big(-n\theta(i/n)+o_p(n)\big),
\]
uniformly for $\delta n\le i\le(1-\delta)n$, for every fixed
$\delta\in(0,1/2)$. The coefficients near the endpoints require
separate estimates.

The profile $\theta$ can be read directly from the representation.
Lemma~\ref{lem:R_conc} gives $R/n\to1/2$ in probability.
Replacing the exponential variables in Lemma~\ref{lem:rep} by
their means and passing from sums to integrals gives the centered
profile
\[ \tilde\theta(t) =\int_0^{|t|}\frac{2(|t|-u)}{(1/2-u)^2} du
=-2\log(1-2|t|)-4|t|, \qquad |t|<1/2. \]
Thus $\theta(s)=\tilde\theta(s-1/2)$ for $0<s<1$.
This profile is strictly convex, has its minimum at $s=1/2$,
and diverges at both endpoints.

Related results for random polynomials of the form
$\sum_{i=0}^n\eps_i f_{n,i}z^i$ were obtained in
\cite{kabluchko2014asymptotic,bloom19}. In a basic setting, the
$\eps_i$ are nondegenerate i.i.d.\ random variables satisfying
$\E\log(1+|\eps_0|)<\infty$, and the deterministic coefficients
$f_{n,i}$ satisfy
\[ \lim_{n\to\infty}\max_{0\le i\le n} \big||f_{n,i}|^{1/n}-f(i/n)\big|=0
\]
for some positive continuous function $f$ on $[0,1]$.
The limiting root distribution is then determined by this
deterministic exponential profile. Our model differs from this
setting: the random coefficient fluctuations are dependent and
their joint law varies with $n$, while the corresponding profile
$e^{-\theta(s)}$ tends to zero at both endpoints.

As in \cite{kabluchko2014asymptotic,bloom19}, we establish convergence
of the empirical root measures through convergence of the normalized
logarithmic potentials. For each fixed
$z\in\mathbb C\setminus\{0\}$, we prove that
$n^{-1}\log|P_n(z)|$ converges in probability to
\[ G(z)=\sup_{0<s<1} \big( s\log|z|-\theta(s) \big).
\]
The upper bound follows from the coefficient estimates and the
triangle inequality. The main additional difficulty is the lower
bound, which requires ruling out cancellations on the exponential
scale despite the dependence among the coefficients. To do this,
we condition on $R$, on all but three neighboring variables in the
exponential representation, and on two suitable linear combinations
of those three variables. This leaves one polynomial coefficient
random while fixing all the others. A conditional anti-concentration
estimate then gives the required lower bound.

\subsection{Open problems and future direction}

The results above concern the empirical root measure and do not
determine the asymptotic number of real zeros. We now ask whether
conditioning on log-concavity changes this finer statistic, and
how the answer depends on $\alpha$.

Recall the $\alpha$-power  model $P_{n,\alpha}$ from
\eqref{eq:b=a^n} and its unconditioned counterpart $Q_{n,\alpha}$
from \eqref{def:unconditioned_alpha_model}. Write $N_{\mathbb R}(f)$
for the number of real zeros of a polynomial $f$, counted with
multiplicity. Since both models have strictly positive coefficients,
all their real zeros are negative.

For the unconditioned model, the coefficients are exchangeable.
Consequently, Corollary~2 of \cite{soze2017real} gives
\begin{equation}\label{eq:unconditioned_real_root_bound}
    \mathbb E N_{\mathbb R}(Q_{n,\alpha})
    \le C\log n,
    \qquad n\ge2,\quad \alpha\ge 0,
\end{equation}
where $C$ is an absolute constant. Thus, the expected number of real roots in the unconditioned model satisfies the same logarithmic-order upper bound as for the classical Kac polynomial.
When $\alpha=0$,
the classical expectation asymptotics \cite{ibragimov1971nonzero},
together with the variance estimate in
\cite{do2026noncentered}, imply that $Q_{n,0}$ has
$\Theta(\log n)$ real zeros with high probability.

\begin{problem}\label{prob 1}
For fixed $\alpha\in[0,1]$, determine the asymptotic order of
$N_{\mathbb R}(P_{n,\alpha})$, both in expectation and with high
probability. Does conditioning on log-concavity increase the expected
number of real zeros for all sufficiently large $n$, and how does
this comparison depend on $\alpha$? For the unconditioned model, does $Q_{n,\alpha}$ have
$\Theta(\log n)$ real zeros with high probability for $\alpha\in(0,1]$?
\end{problem}
For fixed $\alpha\in(0,1]$, we suspect that  
\[ \mathbb E N_{\mathbb R}(P_{n,\alpha}) \gtrsim  n^{\alpha/2} .\]
In particular, for the beta model $\alpha=1$, this predicts at least order $\sqrt n$ real roots in expectation.  On the other hand, for the uniform
log-concave model $\alpha=0$, we expect that
\[ \mathbb E N_{\mathbb R}(P_{n,0}) = \Theta(\log n). \]

Finally, it would be interesting to investigate the microscopic
root statistics of the $\alpha$-power model, including
correlation functions and local repulsion, both near the real axis
and away from it.

The rest of the paper is organized as follows. In Section~2, we derive a convenient representation of the random convex sequence. In Section~3, we prove convergence to the uniform measure on the unit circle for $P_{n,\alpha}$ when $0\le\alpha<1$ and for $Q_{n,\alpha}$ when $0\le\alpha\le1$. In Section~4, we prove Theorem~\ref{thm:main_beta}.

\section{Random convex sequence}
In this section, we give a more convenient description of the joint distribution of the random convex sequence $ \xi_0, \xi_1, \ldots, \xi_n$ defined in \eqref{def:random_convex}, expressed in terms of independent exponential random variables.

Let $X_0, X_1, \ldots, X_n$ be i.i.d.\ $\mathrm{exp}(1)$ random variables.
Define
\[ Y_i := \Delta_i X =  X_i - X_{i-1}, \ \ 1 \le i \le n, \ \  \ Z_i  :=\Delta_i Y =   Y_i - Y_{i-1} =  \Delta^2_i X, \ \ 2 \le i \le n. \] 
Recall that $ \mathcal{C} =\{ Z_i  > 0,  \ \ 2 \le i \le n \}.$ 
 %\[  ( \xi_ 0, \xi_1, \xi_2, \ldots, \xi_n) \stackrel{d}{=} (X_0, X_1, \ldots, X_n ) \big \vert \mathcal{C}.\]
Denote
\[
I=\min\{ i\in\{0,1,\dots,n-1\}: Y_{i+1}>0\},
\]
with the convention that $I=n$ if all $Y_1,\dots,Y_n\le 0$. 
Note that the conditional distribution of $X_0, \ldots, X_n$ given $\mathcal{C}$ can be expressed as the following mixture 
\[ \prob \big( (X_0, \ldots, X_k) \in \cS\ |\ \mathcal{C} \big) = \sum_{i=0}^n \prob \big( (X_0, \ldots, X_k) \in \cS\ |\ \mathcal{C} \cap \{ I = i\} \big)\prob(I = i\ |\  \mathcal{C}). \]

\begin{lem}\label{lem:rep}
Let $\{E_j\}_{j=-n}^{n}$ be i.i.d.\ $\exp(1)$ random variables,
and let $R$ be an independent random variable whose probability mass function is given by 
\begin{equation} \label{eq:R_pmf}
 \prob(R=i) = \frac{n+2}{n+1}\cdot\frac{\binom{n+1}{i}\binom{n+1}{i+1}}{\binom{2n+2}{ n+1}}, \quad i = 0, 1, 2, \ldots, n. 
\end{equation}
Writing $T_r = r(r+1)/2$, denote
\begin{align*}
W_{i-k} &= 
\frac{1}{n+1}E_0 +\sum_{m=1}^{k}\frac{k-m+1}{T_{i-m+1}} E_{-m}, & 0\le k\le i, %\label{eq:ind_x_below_express}
\\
W_{i+k} &= \frac{1}{n+1}E_0 +\sum_{m=1}^{k}\frac{k-m+1}{T_{n-i -m+1}} E_{m}, & 1\le k\le n-i.
%\label{eq:ind_x_above_express}
\end{align*}
Then
\begin{align*}(W_{0},\dots, W_{n})&\overset{\mathrm{d}}= (X_0, X_1, \dots, X_n)\ |\ \mathcal{C}, I=i.\\
R&\overset{\mathrm{d}}=I\ |\ \mathcal C.
\end{align*}
\end{lem}
\begin{rem}\label{rem:W-convex}
From the representation in Lemma~\ref{eq:R_pmf}, it follows that for $1\le k\le R$,
\[
W_{R-k}-W_{R-(k-1)}=\sum_{m=1}^{k}\frac{1}{T_{R-m+1}} E_{-m}>0,
\quad
W_{R-k}-2W_{R-(k-1)}+W_{R-(k-2)}=\frac{1}{T_{R-k+1}} E_{-k}>0,
\]
and for $1\le k\le n-R$,
\[
W_{R+k}-W_{R+(k-1)}=\sum_{m=1}^{k}\frac{1}{T_{n-R-m+1}} E_{m}>0,
\quad
W_{R+k}-2W_{R+(k-1)}+W_{R+(k-2)}=\frac{1}{T_{n-R-k+1}} E_{k}>0.
\]
In particular, $(W_k)$ is strictly decreasing on $\{0,\dots,R\}$ and strictly increasing on $\{R,\dots,n\}$, so $W_R$
is the unique global minimum, and the profile is $V$-shaped around $R$. Moreover, $(W_k)$ is strictly convex on each side. Equivalently, $(e^{-W_k})_{k=0}^n$ is strictly log-concave.
\end{rem}

\begin{rem}
Define Catalan and Narayana numbers as 
 $C_m=\frac{1}{m+1}\binom{2m}{m}$  and $N(m,k)=\frac{1}{m}\binom{m}{k}\binom{m}{k-1}$. These count the number of Dyck paths of length $m$ and the number of Dyck paths of length $m$ with exactly $k$ peaks, respectively. Then 
 \[ \prob(R = i) = \frac{N(n+1,i+1)}{C_{n+1}},  \ \ 0 \le i \le n. \]
\end{rem}

\begin{proof}
Via a change of variables through $X_k=X_0+\sum_{t=1}^kY_t$ 
we compute the joint density of $(X_0,Y_1,\dots,Y_n)$ as
\[
f_{(X_0,Y_1,\dots,Y_n)}(x_0, y_1, \ldots, y_n) = 
\exp \Big(-(n+1)x_0-\sum_{k=1}^n (n+1-k)y_k\Big) \mathbf{1}_{\{ x_0 + \sum_{t=1}^k y_t \ge 0 \text{ for all } k\ge 0 \} }.
\]
On $\mathcal C\cap\{I=i\}$ we have 
\[
y_1<\cdots<y_i<0<y_{i+1}<\cdots<y_n.
\]
writing $V_j := \begin{cases}
Y_{i+j+1} & j<0,\\
Y_{i+j} & j>0,
\end{cases}$
the join density of $(V_{-i},\dots,V_{-1},X_i,V_{1},\dots,V_{n-i})$ satisfies
\begin{multline*}
  f_{(V_{-i},\dots,V_{-1},X_i,V_{1},\dots,V_{n-i})}(v_{-i}, \ldots, v_{-1},x_i,v_{1},\dots,v_{n-i})=\\
\exp \Big(-(n+1)x_i-\sum_{k=-1}^{-i} (-i-1-k)v_k-\sum_{k=1}^{n-i} (n-i+1-k)v_k\Big)\mathbf{1}_{\left\{\substack{ x_i + \sum_{t=1}^k v_t \ge 0 \text{ for all } k\in [0,n-i], \\\hspace{6pt}x_i + \sum_{t=-1}^{k} v_{t} \ge 0 \text{ for all } k\in [-i,-1] }\right\} }.
\end{multline*}
on $\mathcal{C},\{I=i\}$, this implies 
\begin{multline*}
\mathbf{1} (C\cap\{I=i\})f(v_{-i}, \ldots, v_{-1},x_i,v_{1},\dots,v_{n-i})\propto\\
\exp \Big(-(n+1)x_i-\sum_{k=-1}^{-i} (-i-1-k)v_k-\sum_{k=1}^{n-i} (n-i+1-k)v_k\Big)\mathbf{1}_{\left\{\substack{x_i,v_1,-v_{-1}\ge 0,\\v_k\le v_{k+1} \text{ for all } k\in [-i,n-i-1]}\right\}},
\end{multline*}
as required. 

Next, write $U_k=\begin{cases}
    V_k\cdot \mathrm{sgn}(k) &\text{if }k\in\{\pm1\}\\
    V_k-V_{k-1}&\text{if }k\notin\{\pm1\}
\end{cases}$. 
on $\mathcal{C},\{I=i\}$, we can express $X_0, X_1, \ldots, X_n$ as the linear combinations of $X_i$ and $U_{-i}, \ldots, U_{-1}, U_{1}, \ldots, U_{n-i}$ via the following identities. 
\begin{align}
X_{i-j} &= X_i+\sum_{m=1}^{j}(j-m+1) U_{-m}, && 0\le k\le i, \label{eq: ind_x_below_express}\\
X_{i+j} &= X_i+\sum_{m=1}^{j}(j-m+1)\,U_{m}, && 0\le k\le n-i. \label{eq: ind_x_above_express}
\end{align}
Writing $T_r = r(r+1)/2$, this implies
\begin{align}%\label{eq:sum_x_express}
\sum_{j=0}^n X_j
= (n+1)X_i + \sum_{m=-1}^{-i} T_{1+i-m} U_m
+ \sum_{m=1}^{n-i} T_{n-i+m+1} U_m.
\end{align}
Using this we compute the joint density of $(X_i,U_{-i}, \ldots, U_{n-i})$, 
up to a constant, exploiting the fact that the change of variables has a constant Jacobian,
\begin{multline}
f_{(X_i,U_{-i}, \ldots, U_{-1},U_{1},\dots, U_{n-i})}(x_i,u_{-i}, \ldots, u_{-1},u_{1},\dots, u_{n-i})\propto\\
   \exp \Big( - (n+1)x_i  
- \sum_{m-1}^i T_{1+i-m} u_{-m}
 -  \sum_{m=1}^{n-i} T_{n-m-i+1} u_m \Big) \cdot \mathbf{1}_{\{x_k > 0 \text { for all }  k\} }.  \label{eq: main obs on U}
\end{multline}
This we use twice to obtain the statement of the lemma. Firstly, we obtain that 
\begin{align*}
\prob(I=i)
&\propto\int_{(0,\infty)}\exp\left(-\frac{x_i}{n+1}\right)dx_i\int_{(0,\infty)^i}\exp\!\Big(-\sum_{j=1}^{i}T_j u_{-i+j-1}\Big)\,d(u_{-i},\dots, u_{-1})
\\&\phantom{=}\;\cdot\;
\int_{(0,\infty)^{\,n-i}}\exp\!\Big(-\sum_{j=1}^{n-i}T_j u_{n-i-j+1}\Big)\,d(u_1,\dots, u_{n-i})\\
&\propto\prod_{j=1}^{i}\frac{1}{T_j}\;\cdot\;\prod_{j=1}^{n-i}\frac{1}{T_j},
\end{align*}
Where for the last step we integrate out $x_i$. We have $\prod_{j=1}^{r}T_j=2^{-r}r!(r+1)!$, so
\[
w_i=\frac{2^n}{\,i!(i+1)!\,(n-i)!(n-i+1)!}.
\]
Finally, we observe that 
\[
\prob(R=i)=\frac{w_i}{\sum_{r=0}^{n} w_r}
=\frac{\frac{1}{i!(i+1)!\,(n-i)!(n-i+1)!}}
{\sum_{r=0}^{n}\frac{1}{r!(r+1)!\,(n-r)!(n-r+1)!}}.
\]
Multiply numerator and denominator by $(n+1)!^2$ and use the identity
\[
\sum_{r=0}^{n}\binom{n+1}{r}\binom{n+1}{r+1}
=\frac{n+1}{n+2}\binom{2n+2}{n+1}.
\]
This yields   \eqref{eq:R_pmf}.

In addition, in terms of \eqref{eq: main obs on U}, the event $\mathcal{C}\cap \{I=i\}$ becomes $\{x_i>0\}, \{u_m>0\}_{m\in \{-i,\dots, -1,1,\dots, n-i\}}$. Conditioned on these the formulae
\eqref{eq: ind_x_below_express}–\eqref{eq: ind_x_above_express} yield $x_k>0$ for all $k$. Therefore, conditioning on $\mathcal C\cap\{I=i\}$ we have, 
\[
X_i\sim \mathrm{exp}(n+1),\qquad
U_m\sim \mathrm{exp}(T_{\,1+i-m}),\ \ -i\le m\le 1, \qquad
U_m\sim \mathrm{exp}(T_{\,n-m-i+1}),\ \ 1\le m\le n-i,
\]
and all the listed variables are mutually independent. 
Noting that $\mathrm{exp}(\lambda)$ has the same distribution as $\lambda^{-1}\mathrm{exp}(1), $ for any $\lambda>0,$ we express 
for $0\le k\le R-1$,
\[
X_k
= \frac{1}{n+1} X_i + \sum_{m=1}^{k}
\frac{m - k -1}{ T_{m-1}} U_{-m},
\]
and for $R \le k \le n$,
\[
X_k
= \frac{1}{n+1}X_i +\sum_{m=R+1}^{k}\frac{k - m +1}{ T_{n-m+1}} U_m.
\] 
The lemma follows.
\end{proof}

We also require the following concentration 
bound on $R$.
 \begin{lem}\label{lem:R_conc}
     Let $R$ be as defined in \eqref{eq:R_pmf}. Then there exist constants $C, c>0$ such that for all $n \ge 1$ and $t > 0,$
     \[ \prob( |R - n/2| \ge t ) \le C n e^{-c t^2/n}.\]
 \end{lem}
\begin{proof}
    Note that if $X$ and $Y$ are i.i.d.\ $\mathrm{Bin}(n+1, 1/2),$ then 
    \[ R \stackrel{d}{=} (X \ \big |\ X = Y-1).\] By local central limit theorem, $\prob(X = Y-1) \ge \prob(X = \lfloor{n/2} \rfloor) \prob(Y = \lfloor{n/2} \rfloor -1) = \Omega(n^{-1}).$ The lemma now follows from the standard concentration of the binomial distribution. 
\end{proof}

\section{Convergence to the unit-circle law}

In this section, we first extend Theorem~\ref{thm:uniform}, corresponding to the case $\alpha=0$, by proving convergence to the uniform measure on the unit circle for the $\alpha$-power model $P_{n,\alpha}$ for all $0\le\alpha<1$. We then establish the analogous result for the unconditioned $\alpha$-power model $Q_{n,\alpha}$ for $0\le\alpha\le1$. For the almost-sure statements, we realize the models for all $n$ on a common probability space using triangular arrays. The subcritical results follow from a deterministic equidistribution criterion, while the unconditioned critical case $\alpha=1$ requires a separate logarithmic-potential argument.

\begin{prop}\label{prop:alpha_less_than_one_convergence}
Fix $0\le\alpha<1$. Let $P_{n,\alpha}$ be the random polynomial as defined in \eqref{eq:b=a^n}
and let $\zeta_{n,1},\ldots,\zeta_{n,n}$
be the roots of $P_{n,\alpha}$.
Then, almost surely,
\[
d_L\Big( n^{-1}\sum_{k=1}^n\delta_{\zeta_{n,k}}, \mathsf{U}(\mathbb S^1) \Big)\to 0
\qquad\text{as }n\to\infty.
\]
\end{prop}

The argument is a direct consequence of a result of  \cite{hughes2008zeros}, which provides a sufficient condition on the coefficients of a deterministic polynomial under which its roots equidistribute on the unit circle, that is,  their moduli concentrate at $1$ and their arguments become asymptotically uniform.
\begin{thm}[\cite{hughes2008zeros}]
\label{thm:uniform_dist_root_criteria}
Let $p_n(z)=\sum_{k=0}^n a_{n,k}z^k$ be deterministic
polynomials of degree $n$ with $a_{n,0}\ne0$. If
\begin{equation}\label{uniform_dist_root_criteria}
\frac1n\log\left(
\frac{\sum_{k=0}^n|a_{n,k}|}
{\sqrt{|a_{n,0}a_{n,n}|}}
\right)\longrightarrow0,
\end{equation}
then their empirical root measures converge weakly to
$\mathsf{U}(\mathbb S^1)$.
\end{thm}

\begin{proof}[Proof of Proposition~\ref{prop:alpha_less_than_one_convergence}]
Using Lemma~\ref{lem:rep}, realize $P_{n,\alpha}(z)$ as 
\[ P_{n,\alpha}(z) = \sum_{k=0}^n e^{-n^\alpha W_{n, k}}z^k,
\]
where, for each $n$, $(W_{n, k})_{k=0}^n$ is constructed from
random index $R_n$ and i.i.d.\ $\mathrm{exp}(1)$ random variables
$(E_{n, m})_{|m|\le n}$. Set $M_n=\max_{1\le |m|\le n}E_{n, m}.$
The lemma gives
$W_{n, R_n}=\min_{0\le k\le n}W_{n, k}$ and
\[0\le \frac{W_{n, 0}+W_{n,n}}2-W_{n, R_n}
\le M_n\Big( \sum_{j=2}^{R_n+1}\frac1j+ \sum_{j=2}^{n-R_n+1}\frac1j \Big)
\le 2M_n\log(n+1).\]
Moreover, $\prob(M_n>3\log n) \le 2n e^{-3\log n} = 2n^{-2}.$
Thus, by Borel--Cantelli, $M_n=O(\log n)$ almost surely.  Consequently,
\begin{align*}
0&\le\frac1n\log\left( \frac{\sum_{k=0}^n e^{-n^\alpha W_{n,k}}}
{e^{-n^\alpha(W_{n,0}+W_{n,n})/2}} \right) \\
&\le\frac{\log(n+1)}n+
2n^{\alpha-1}M_n\log(n+1) \to 0
\end{align*}
almost surely, since $\alpha<1$. Hence
\eqref{uniform_dist_root_criteria} holds almost surely, and
Theorem~\ref{thm:uniform_dist_root_criteria} completes the proof.
\end{proof}

The same criterion gives the corresponding result for
the unconditioned $\alpha$-power model for $\alpha \in [0, 1)$.
\begin{prop}\label{prop:unconditioned_alpha_less_than_one_convergence}
For every $0\le\alpha<1$, the empirical root measure of
$Q_{n,\alpha}$ defined in \eqref{def:unconditioned_alpha_model}
converges almost surely to $\mathsf{U}(\mathbb S^1)$.
\end{prop}

\begin{proof}
Write $Q_{n,\alpha}(z)=\sum_{k=0}^n U_{n, k}^{n^\alpha}z^k,$
where for each $n$, $U_{n, 0},\ldots,U_{n,n}$ are i.i.d.\
$U(0,1)$ random variables. Since
$-\log U_{0,n}$ and $-\log U_{n,n}$ are
$\mathrm{exp}(1)$ distributed,
by Borel--Cantelli,  $-\log(U_{n, 0} U_{n, n})=O(\log n)$ almost surely.
As $U_{n,k}^{n^\alpha}\le1$, the quantity in
\eqref{uniform_dist_root_criteria} is at most
\[ \frac{\log(n+1)}{n}-\tfrac12 n^{\alpha-1}\log(U_{n,0}U_{n, n}) \to 0, \]
almost surely, since $\alpha<1$. The result follows from
Theorem~\ref{thm:uniform_dist_root_criteria}.
\end{proof}

When $\alpha=1$, the endpoint term in the preceding estimate need
not tend to zero, so Theorem~\ref{thm:uniform_dist_root_criteria}
does not apply directly. Nevertheless, the unconditioned model $Q_{n, 1}$ still
has the same limiting root distribution. The next result covers the case $1 \le \alpha < 2$.

\begin{prop}\label{prop:unconditioned_beta}
Fix $\alpha \in [1, 2).$
For each $n\ge 1$, let $U_{n, 0},\ldots,U_{n,n}$ be i.i.d.\ $U(0,1)$ random variables, and define 
 $Q_n(z) = \sum_{k=0}^n  U_{n, k}^{n^\alpha} z^k$.
If $\zeta_{n,1},\ldots,\zeta_{n,n}$ are the zeros of $Q_n$, then almost surely
\[
d_L \Big( n^{-1}\sum_{j=1}^n\delta_{\zeta_{n,j}}, \mathsf{U}(\mathbb S^1) \Big) \to 0
\qquad\text{as }n\to\infty.
\]
\end{prop}

\begin{proof}
Set $u_n(z)=n^{-1}\log|Q_n(z)|$. Since $0<U_{n,k}<1$,
\begin{equation}\label{eq:unconditioned_upper}
u_n(z)\le \log^+|z|+\frac{\log(n+1)}n, \qquad z\in\C.
\end{equation}
Let $m_n=\lfloor n^{\alpha/2} \rfloor$. Positivity and independence
of the coefficients give, for every $\eps>0$,
\begin{align*}
\prob\bigl(Q_n(1/2)<2^{-m_n}e^{-\eps n}\bigr) \le
\prob\left(\max_{0\le k\le m_n} U^{n^\alpha}_{n,k}<e^{-\eps n}\right) =e^{-\eps n^{1 - \alpha} (m_n+1)}.
\end{align*}
The bound is summable in $n$. Thus Borel--Cantelli, applied to
positive rational $\eps$, and \eqref{eq:unconditioned_upper}
yield $u_n(1/2)\to0$ almost surely. Applying the same argument
to the reversed polynomial $z^nQ_n(1/z)$ gives
$u_n(2)\to\log2$ almost surely.

Fix a realization on which both limits hold for the rest of the proof. By the
subharmonic compactness theorem
\cite[Theorem~3.2.12]{hormander1994notions},
every subsequence of $(u_n)$ has a further subsequence
$(u_{n_j})$ converging in $L^1_{\mathrm{loc}}(\C)$ to a
subharmonic function $u$. Indeed, local upper boundedness
follows from \eqref{eq:unconditioned_upper}, and
$u_n(1/2)\to0$ excludes locally uniform convergence to
$-\infty$.

Passing to a further subsequence if necessary, the $L^1_{\mathrm{loc}}$ convergence implies $u_{n_j}(z)\to u(z)$ for almost every $z$. Hence \eqref{eq:unconditioned_upper} gives $u(z)\le \log^+|z|$ almost everywhere, and therefore everywhere by standard properties of subharmonic functions. Moreover, Hartogs' lemma
\cite[Theorem~3.2.13]{hormander1994notions} gives
$\limsup_{j\to\infty}u_{n_j}(z)\le u(z)$ for every
$z\in\C$. Using the limits at $1/2$ and $2$, together with
the upper bound, we obtain
\[ u(1/2)=0,\qquad u(2)=\log2. \]

Since $u\le0$ on $\{|z|<1\}$ and attains zero at $1/2$,
the strong maximum principle
\cite[Section~3.2]{hormander1994notions}
gives $u\equiv0$ there.
On the connected domain $\{|z|>1\}$, the function
$u-\log|z|$ is subharmonic because $\log|z|$ is harmonic.
It is nonpositive and vanishes at $2$, so the same
principle gives $u(z)=\log|z|$ throughout this domain.
Thus $u(z)=\log^+|z|$ off the unit circle, and by upper semicontinuity and the upper bound, everywhere on $\C$.

Since the initial subsequence was arbitrary,
the full sequence satisfies
$u_n\to\log^+|\cdot|$ in $L^1_{\mathrm{loc}}(\C)$. Taking distributional Laplacians now gives
\[ \frac1n\sum_{j=1}^n\delta_{\zeta_{n,j}}
=\frac1{2\pi}\Delta u_n
\rightarrow \frac1{2\pi}\Delta\log^+|z| =\mathsf{U}(\mathbb S^1), \]
in the sense of distributions. 
This implies vague convergence, which is weak convergence
because all the measures, including the limit, have total
mass one.
\end{proof}

\section{Convergence of the empirical root measure for the beta model}
By the representation of the beta model in Section~\ref{sec:beta} and Lemma~\ref{lem:rep}, we may realize a random polynomial drawn from the beta model as
\[
P_n(z) \stackrel{d}{=} \sum_{k=0}^n e^{-nW_k} z^k,
\]
where $(W_0,\dots,W_n)$ has the joint distribution given by Lemma~\ref{lem:rep}.

Given $|t| < 1/2$ and $z \in \mathbb{C} \setminus \{0\}$, denote
\begin{equation} \label{def:psi}
\psi(t) = -4|t| - 2\log(1-2|t|), \qquad \psi(t; z) = \psi(t) - (1/2 + t) \log |z|, \ \  |t|<1/2. 
\end{equation}
and set $G(z) := \sup_{ t \in (-1/2, 1/2) }  - \psi(t; z).$ A simple computation yields
\begin{equation}\label{def:V}
G(z) =
\begin{cases}
2 \log \dfrac{4}{4 - \log |z|}, & 0 < |z| \le 1, \\
\log |z| + 2 \log \dfrac{4}{4 + \log |z|}, & |z| \ge 1,
\end{cases}
\end{equation}
It is easy to verify that $G(z)  = -2 \log| \log |z| | + O(1) $ as $z \to 0.$ Hence, we can extend the definition of $G$ to $\mathbb{C}$ by setting $G(0) = - \infty.$
Note that $G$ is a subharmonic function.

Theorem~\ref{thm:main_beta} follows from the following result, which establishes the convergence of the log-potential of $P_n$. 

\begin{thm}\label{thm:conv_log_potential}
Let $P_n(z)$ be a random log-concave polynomial of degree $n$ drawn from the beta model. As $n \to \infty,$ for any $z \in \mathbb{C} \setminus \{0\},$
  \begin{equation}\label{eq:conv_log_potential_main}
     \lim_{n \to \infty}  \frac{1}{n} \log |P_n(z)| \stackrel{p}{\to} G(z),   
  \end{equation}
  where the limiting potential $G$ is as defined in \eqref{def:V}.
  Moreover, for every $C> 1$ and $\eps>0$,
   \begin{equation}\label{eq:conv_log_potential_uniform}
   \lim_{n \to \infty}  \prob \left( \sup_{ C^{-1} \le |z| \le C} \Big( \frac{1}{n} \log |P_n(z)|  - G(z) \Big) \ge \eps  \right) = 0. 
   \end{equation}
\end{thm}

This Theorem is the conjugation of an upper and lower bound of the following form. Fix $z\in\C\setminus\{0\}$ and $\eps>0$ and $C>1$. Then, as $n\to\infty$,
\begin{align}\label{eq:logpotential_ub}
\prob \Big( \sup_{ C^{-1} \le |z| \le C} \Big( \frac{1}{n} \log |P_n(z)|  - G(z) \Big) \le \eps  \Big) \to 1
\end{align}
and
\begin{align}\label{eq:logpotential_lb}
\prob\Big(\frac{1}{n}\log|P_n(z)| \ge G(z)-\eps\Big)\to 1.
\end{align}
The proof of \eqref{eq:logpotential_ub} is postponed to Section~\ref{sec: pf upper}, and that of \eqref{eq:logpotential_ub} -- to Section~\ref{sec: pf lower}.

\medskip
In the remainder of the section, we reduce Theorem~\ref{thm:main_beta} to Theorem~\ref{thm:conv_log_potential}.
To this end, we require the following analytic proposition, which is a slight variant of Theorems~2.4 and~4.1 in \cite{bloom19}. In \cite{bloom19}, the limiting function is assumed to be continuous on $\mathbb{C}$, while in our setting $G$ is continuous only on the punctured plane $\mathbb{C}\setminus\{0\}$. The proof of this version could be found in Appendix~\ref{sec:app root conv}.
\begin{prop}\label{prop:bd_puncture}
Let $p_n$ be (non-random) polynomials of degree $n$ and define
\[
\varphi_n(z)=\frac{1}{n}\log|p_n(z)|, \  z \in \mathbb{C}, \text{ and }
\quad
\mu_{p_n}=\frac{1}{2\pi}\Delta \varphi_n,
\]
where $\mu_{p_n}$ is understood in the sense of distributions. Equivalently, $\mu_{p_n}$ is the empirical measure of the roots of $p_n$.  Let $U: \mathbb C \to \mathbb{R} \cup \{ -\infty\} $ be subharmonic, assume that
$U$ is continuous on $\mathbb C\setminus\{0\}$, and set
$\mu=\frac{1}{2\pi}\Delta U$ (the distributional Laplacian), viewed as a Radon measure on $\mathbb C$.  

Suppose that the following hold.
\begin{enumerate}[label=\textup{(\roman*)}, ref=\textup{(\roman*)}]
\item The family $\{\varphi_n\}$ is locally bounded above on $\mathbb C$.\label{prop:bd_puncture:it1}
\item  For every $z\in\mathbb C\setminus\{0\}$,
\[
\limsup_{n\to\infty}\varphi_n(z)\le U(z).
\] \label{prop:bd_puncture:it2}
\item There exists a countable dense set
$\{z_i\}_{i\ge 1}\subset \mathbb C\setminus\{0\}$ such that
\[
\lim_{n\to\infty}\varphi_n(z_i)=U(z_i)\qquad\text{for all }i\ge 1.
\]
\label{prop:bd_puncture:it3}
\end{enumerate}
Then  for each $r>0$, 
    \begin{equation}\label{L1_loc_conv}
        \varphi_n \to U \qquad\text{in }L^1_{\mathrm{loc}}(E_r), 
 \end{equation}
where $D(0, r) = \{ z: |z| < r \}$ and  $E_r=\mathbb C\setminus \overline{D(0,r)}$.

Furthermore, assume that $\mu(\{0\})=0$ and $\mu(\mathbb{C}) =1$. Also, assume that 
\begin{equation}\label{eq:assump_no_accumulation_zero_origin}
    \lim_{r\downarrow 0}\ \limsup_{n\to\infty} \mu_{p_n}\big(D(0,r)\big)=0.
\end{equation}
 Then $\mu_{p_n}\Rightarrow \mu$ as  weak convergence of probability measures on $\mathbb{C}$.
\end{prop}

\noindent {\bf Proof of Theorem~\ref{thm:main_beta}.}
First, we compute the Riesz measure $\mu$ associated with the subharmonic function $G$ given in \eqref{def:V}.
Let $v:(0,\infty)\to\mathbb R$ be the radial profile $v(r)=G(z)$ with $r=|z|$.  Thus
\[
v(r)=
\begin{cases}
2\log\!\dfrac{4}{4-\log r}, & 0<r\le 1,\\
\log r+2\log\!\dfrac{4}{4+\log r}, & r\ge 1.
\end{cases}
\]
On each region $(0,1)$ and $(1,\infty)$ the function $v$ is $C^\infty$, and for a radial function one has
\[
\Delta G(z)=v''(r)+\frac{1}{r}v'(r),\qquad r=|z|.
\]
A direct computation gives
\[
v'(r)=\frac{2}{r(4-\log r)},
\qquad
v''(r)+\frac{1}{r}v'(r)=\frac{2}{r^2(4-\log r)^2}, \quad 0 < r < 1,
\]
and
\[
v'(r)=\frac{\log r+2}{r(4+\log r)},
\qquad
v''(r)+\frac{1}{r}v'(r)=\frac{2}{r^2(4+\log r)^2}, \quad 1 < r < \infty.
\]
Hence, both for $0<|z|<1$ and $|z|>1$
\begin{equation}\label{density_mu}
  \frac{1}{2\pi}\Delta G(z)=\frac{1}{\pi\,|z|^2\big(4 +|\log|z||\big)^2}.
  \end{equation}
At $r=1$ the two formulas agree and $v$ is $C^1$ across $r=1$ since
\[
v(1^-)=v(1^+)=0,
\qquad
v'(1^-)=\frac{2}{4}= \frac12,
\qquad
v'(1^+)=\frac{2}{4}=\frac12.
\]
Therefore the distributional Laplacian $\Delta G$ has no additional measure supported on $\{|z|=1\}$, and the measure $\mu= (2\pi)^{-1}\Delta G$
is absolutely continuous on $\mathbb C\setminus\{0\}$ with density against planar Lebesgue measure as given in \eqref{density_mu}.
For $0<\delta<1$, using polar coordinates and the above density,
\[
\mu\big(D(0,\delta)\big)
=\int_{0}^{\delta}\int_{0}^{2\pi}\frac{1}{\pi r^2(4-\log r)^2} r d\theta dr
=\int_{0}^{\delta}\frac{2}{r(4-\log r)^2}dr =  \frac{2}{4-\log \delta}.
\]
In particular, 
\[
\mu(\{0\})=\lim_{\delta \downarrow 0}\mu\big(D(0,\delta)\big)=0.
\]
This verifies that $\mu$ is indeed the measure given in the statement of the theorem.

Next, we show that $\mu(\mathbb C)=1$. For $t>1$, Green's identity  gives
\[
\mu(D(0,t))=\frac{1}{2\pi}\int_{|z|\le t }\Delta G dz
=\frac{1}{2\pi}\int_{|z|=t}\partial_{\mathsf n} G ds
=t v'(t) =  \frac{\log t +2}{(4+\log t)},
\]
where $\partial_{\mathsf n}$ is the outward normal derivative and $ds$ is the arclength measure on the circle. Letting $t\to\infty$ yields
\[
\mu(\mathbb C)=\lim_{t \to\infty}\mu(D(0,t))=1.
\]

To establish \eqref{eq:main_1_2}, it suffices to verify that every subsequence has a further subsequence along which the displayed convergence holds almost surely. 
Accordingly, fix an arbitrary subsequence, $(n_k)$. 
We now aim to obtain convergence of the empirical measure directly from Proposition~\ref{prop:bd_puncture}.

To conclude the proof, we must therefore the assumptions of the proposition for the family $\varphi_k={n_k}^{-1}  \log |P_{n_k}(z)|$. 

Since each $W_k$ is non-negative, 
\begin{align*}
   |P_n(z)| =  \Big|\sum_{k=0}^n e^{- n W_k} z^k\Big| \le (n+1)(1 + |z|)^n.
\end{align*}
Hence, $  n^{-1}  \log |P_n(z)| \le n^{-1} \log (n+1) + 1+ |z|.$ Consequently, 
$\{ n^{-1}  \log |P_n(z)|\}_{n \ge 1}$ is locally bounded above almost surely. This establishes assumption~\ref{prop:bd_puncture:it1}.

Next, iteratively apply Theorem~\ref{thm:conv_log_potential} for $C_k=\frac{k+1}{k}$ and $\epsilon_k=\frac{1}{k}$, to construct nested subsequences satisfying \eqref{eq:conv_log_potential_uniform}. Diagonalize over this to obtain a subsequence, which, with a slight abuse of notation we keep denoting by $n_k$, such that
\[
\limsup_{k\to\infty}\frac{1}{n_k}\log |P_{n_k}(z)| \le G(z), \qquad z\ne 0.
\]
This establishes assumption~\ref{prop:bd_puncture:it2}. Without loss of generality we assume $(n_k)$ was the original sequence and continue to denote it by $(n)$.  

Now, fix a countable dense set $\mathcal D\subset \mathbb C\setminus\{0\}$. 
Iteratively apply Theorem~\ref{thm:conv_log_potential} for $C_k=\frac{k+1}{k}$ and $\epsilon_k=\frac{1}{k}$, 
at each point $z\in\mathcal D$  to construct further nested subsequences such that
\[
\lim_{n_k\to\infty}n_k^{-1}  \log |P_{n_k}(z)|=G(z).
\]
Diagonalizing once again, this establishes assumption~\ref{prop:bd_puncture:it3}. 

It remains to verify that \eqref{eq:assump_no_accumulation_zero_origin} holds almost surely along $(n_k)$.

First note that \eqref{L1_loc_conv} implies that the following almost sure convergence of the circular mean for any $r \in \{q: q \in \mathbb{Q}_+ \} \cup \{\sqrt{q}: q \in \mathbb{Q}_+ \},$ where  $\mathbb{Q}_+$ denotes the set of positive rationals: 
\begin{equation}\label{circular_mean_conv}
\lim_{n\to\infty} m_{n}(r) := \lim_{n\to\infty}\frac{1}{2\pi}\int_0^{2\pi}\frac{1}{n}\log\big|P_{n}(re^{i\theta})\big| d\theta
=
\frac{1}{2\pi}\int_0^{2\pi}G(re^{i\theta})\,d\theta
=:m(r).
\end{equation}
Let $\{\zeta_{n_k, j}\}$ be its zeros (counted with multiplicity) of $P_{n_k}$.
Jensen's formula (centered at $0$) states
\[
m_{P_{n_k}}(r)=\log|P_{n_k}(0)|+\sum_{|\zeta_{n_k, j}|<r}\log\frac{r}{|\zeta_{n_k, j}|}.\]
Applying this with $r=\sqrt{\delta}$ and $r=\delta$ with  $\delta \in \mathbb{Q}_+ \cap (0, 1)$, and subtracting, we get
\[
m_{P_{n_k}}(\sqrt{\delta})-m_{P_{n_k}}(\delta)
=
\sum_{|\zeta_{n_k, j}|<\delta}\log\frac{\sqrt{\delta}}{\delta}
+\sum_{\delta\le |\zeta_{n_k, j}| <\sqrt{\delta}}\log\frac{\sqrt{\delta}}{|\zeta_{n_k, j}|}.
\]
The second sum is nonnegative, hence
\[
m_{P_{n_k}}(\sqrt{\delta})-m_{P_{n_k}}(\delta)
\ge N_{P_{n_k}}(D(0,\delta)) \log\frac{1}{\sqrt{\delta}},
\]
where $N_{P_{n_k}}(D(0,\delta)) $ denotes the number of zeros of $P_{N_k}$ in $D(0,\delta)$, counted with multiplicity. Therefore,
\[
N_{P_{n_k}}(D(0,\delta)) 
\le
\frac{m_{P_{n_k}}(\sqrt{\delta})-m_{P_{n_k}}(\delta)}{\log(1/\sqrt{\delta})}.
\]
By \eqref{circular_mean_conv}, after letting $k\to\infty$ in the Jensen bound, we obtain almost surely
\[
\limsup_{k\to\infty}\mu_{P_{n_k}}(D(0,\delta))
\le
\frac{m_{\sqrt{\delta}}-m_\delta}{\log(1/\sqrt{\delta})}.
\]
Since $G$ is radial, we have
\[
m_r=\frac{1}{2\pi}\int_0^{2\pi}G(re^{i\theta}) d\theta=v(r),
\]
and we already observed $v(r) =  - 2 \log| \log r | + O(1)$ as $r\to  0+$.
This implies that $m_{\sqrt{\delta}}-m_\delta=O(1)$ as $ \delta\to  0+$
and therefore
\[
\frac{m_{\sqrt{\delta}}-m_\delta}{\log(1/\sqrt{\delta})}\to 0
\quad\text{as }\delta\to 0+.
\]
Hence, \eqref{eq:assump_no_accumulation_zero_origin} holds almost surely along $(n_k)$, which yields Theorem~\ref{thm:main_beta}. \qed

\subsection{Proof of the upper bound}\label{sec: pf upper}
This Section is devoted to the proof of \eqref{eq:logpotential_ub}.

Recall $W_i$ and $R$ from Lemma~\ref{lem:rep}, and define
\begin{align*}
\ccW_{R-k} &= W_{R-k}-\E(W_{R-k}\ |\ R)=
\frac{1}{n+1}(E_0-1) +\sum_{m=1}^{k}\frac{k-m+1}{T_{R-m+1}} ( E_{-m} - 1), & 0\le k\le R, 
\\
\ccW_{R+k} &= W_{R+k}-\E(W_{R+k}\ |\ R)=\frac{1}{n+1}(E_0-1) +\sum_{m=1}^{k}\frac{k-m+1}{T_{n-R -m+1}} (E_{m}-1), & 1\le k\le n-R.
\end{align*}

Fix $z \in \mathbb{C} \setminus \{0\}$. By the triangle inequality, 
\begin{align*}
    |P_n(z)| &\le \sum_{\ell=0}^n e^{- n W_\ell } |z|^{\ell} \le (n+1) \max _{0 \le \ell \le n}  \exp \big( - n  \big( W_l - (\ell/n) \log |z|   \big)  \big).
\end{align*}
Therefore, 
\begin{align}%\label{logP_ub1}
    \frac{1}{n} \log |P_n(z)| &\le  \frac{1}{n} \log(n+1)  -  \min _{-R  \le i \le n-R}  \Big ( W_{R+i} - \frac{R+i}{n} \log |z| \Big).
\end{align}
Hence, recalling that $G(z)=\sup_{ t \in (-1/2, 1/2) }  - \psi(t; z)$, in order to achieve our goal 
 \eqref{eq:logpotential_ub},  it would suffice to 
show that, for any fixed $C>1$ and $\eps>0$, the following holds with probability tending to one as $n\to\infty$.
\begin{equation}\label{eq:up goal}
\min_{-R \le i \le n-R}\Big( W_{R+i} - \frac{R+i}{n}\log |z| \Big)
\ge \inf_{t\in (-1/2,1/2)}\psi(t;z) - \eps,
\quad \text{for all } C^{-1} \le |z| \le C,
\end{equation}

Define
\begin{equation}\label{def:psi_n}
\psi_n(i/n) = \E\bigl[W_{R+i}  |\ R\bigr], \qquad 
\psi_n(i/n; z) =  \psi_n(i/n)   - \frac{R+i}{n} \log |z|, \ \ -R \le i \le n- R.   
\end{equation}

We require three Lemmata, whose proofs we postpone to the end of this section. The first shows that $\psi_n(i/n; z)$ well approximates $\psi(i/n; z)$.

\begin{lem}\label{lem:psi_n_limit}
For any $C>1$, 
\[
\sup_{ C^{-1} \le |z| \le C}\Big|
\min_{i \in \cB_n }\psi_n(i/n; z)
-
\inf_{t\in (-1/2,1/2)}\psi(t; z)
\Big|
\xrightarrow{p} 0
\quad\text{as }n\to\infty,
\]
where $\psi_n(t; z)$ and $\psi(t; z)$ are as defined in \eqref{def:psi_n} and \eqref{def:psi} respectively.
\end{lem}

Next, we separately consider coefficients in the bulk and near the edges. Denote $L_n:=\lfloor(\log n)^2\rfloor$,
\[
\cB_n  = ([-R+L_n, n-R-L_n])\cap\mathbb Z, \qquad \cE_n=\bigl([-R,-R+L_n]\cup[n-R-L_n,n-R]\bigr)\cap\mathbb{Z},
\]
so that $[-R, n-R]\cap \mathbb Z = \cB_n \cup \cE_n.$
Using Lemma~\ref{lem:psi_n_limit}, we lower bound the left-hand-side of \eqref{eq:up goal}.
\begin{align}\label{eq: up goal2}
     \min _{-R  \le i \le n-R}  \Big ( W_{R+i} - \frac{R+i}{n} \log |z| \Big) \ge \min \Big(  \min_{i \in \cE_n}  W_{R+i}  - |\log |z||,  \   \min_{i \in \cB_n} \psi_n(i/n; z) - \max_{ i \in \cB_n} |\ccW_{R+i}| \Big).
\end{align}

The next couple of lemmata bound these separately.
\begin{lem}\label{lem:eta_bulk_small}
As $n\to\infty$,
\[
\prob \Big ( \max_{ i \in \cB_n} \  | \ccW_{R+i}| \le (\log n)^{-1/3} \Big) \to 1.
\]
\end{lem}

\begin{lem}\label{lem:edge}
For all $K >0$,
\[
\lim_{n\to\infty}\mathbb{P}\Bigl(\min_{i\in \cE_n} W_{R+i}\ge K \Bigr)=1.
\]
\end{lem}

 Indeed, applying Lemma~\ref{lem:eta_bulk_small}, and Lemma~\ref{lem:edge} with \[
K:= \log C + \sup_{C^{-1} \le |z| \le C}\inf_{t\in (-1/2,1/2)}\psi(t;z) + 1,
\]
 reduced \eqref{eq:up goal} to
 \eqref{eq: up goal2}, concluding the proof. \qed

\medskip
The remainder of the section is devoted to the proof of Lemmata~\ref{lem:psi_n_limit},~\ref{lem:eta_bulk_small} and~\ref{lem:edge}.
\subsubsection{Convergence of expectations -- proof of Lemma~\ref{lem:psi_n_limit}}
Let $C,\varepsilon>0$, and denote $M=\log C$.
Given $\theta>0$, denote the sequence of events
\[
A_n(\theta):=\Bigl\{\Bigl|\frac{R}{n}-\frac12\Bigr|\le \theta\Bigr\},
\]
and observe that by  Lemma~\ref{lem:R_conc}
we have $\prob(A_n(\theta))\to1$. For all  $n$ sufficiently large, we have $L_n/n\le\theta$. Denoting
\begin{align*}
B(\theta)&:=[-1/2+2\theta,\,1/2-2\theta],&
B^c(\theta)&:=[-1/2,\,1/2]\setminus B(\theta),\\
B_n(\theta)&:=B(\theta)\cap \tfrac1n\mathbb Z,& \mathcal T_n&:=(1/n)\cB_n,
\end{align*}
this implies $\{B_n(\theta)
 \subseteq \mathcal T_n\} \supset A_n(\theta)$.
 Using $\lim_{|t| \to 1/2} \psi(t)  = \infty,$ we   
choose $\theta\in(0,1/8)$ sufficiently 
small to satisfy
\begin{equation}\label{eq: theta_choices}
\inf_{\,|t|\in B^c(\theta)}\psi(t)\ge \inf_{|t|<1/2}\psi(t)+M+3\varepsilon,
\qquad
\log\Big(\frac{1/2-\theta}{4\theta}\Big)-1\ge \inf_{|t|<1/2}\psi(t)+(1+\theta)M+2\varepsilon.
\end{equation}
%For $t\in\mathcal T_n\setminus B_n(\theta)$ we have $|t|\ge 1/2-2\theta$, and on $A_n(\theta)$ also $\mathcal T_n\subset[-1/2-\theta,\,1/2+\theta]$, hence $|t|\le 1/2+\theta$ for all $t\in\mathcal T_n$.

Next, we relate $\psi_n(t)$ and $\psi(t)$ within $B_n(\theta)$.
\begin{claim}\label{cl: control psi}
There exists a constant $C_\theta>0$ such that for all $n$ sufficiently large,
\begin{equation}\label{eq:bulk_unif_self}
\left\{\max_{t\in B_n(\theta)}|\psi_n(t)-\psi(t)|
\le C_\theta\Big(\Big|\frac{R}{n}-\frac12\Big|+\frac1n\Big)\right\}\supset A_n(\theta), 
\end{equation}    
\end{claim}
\begin{proof}
Let $t\in[-1/2+2\theta,0]\cap(1/n)\mathbb Z$ and write $\ell:=-nt$. Then re-indexing $s = \ell - m+1$, we have
\begin{equation}\label{rep:psi_negative_l}
  \psi_n(t)=\E[W_{R-\ell}\mid R]= \frac{1}{n+1}+\sum_{m=1}^{\ell}\frac{\ell -m+1}{T_{R-m+1}}  = \frac{1}{n+1}+\sum_{s=1}^{\ell}\frac{s}{T_{R-\ell+s}}.  
\end{equation}
On $A_n(\theta)$ we have 
$R-\ell\ge \theta n$ so that $u:=R-\ell+s \in [\theta n, n] $, implying  $u\asymp_\theta n$ for $1\le s\le \ell$. Observing that 
$1/T_u=2/(u(u+1))=2/u^2+O_\theta(n^{-3})$, and plugging this into \eqref{rep:psi_negative_l} gives
\[
\psi_n(t)=\sum_{s=1}^{\ell}\frac{2s}{(R-\ell+s)^2}+O_\theta(n^{-1}).
\]
As approximating the sum by the integral of $f(x)=2x/(R-\ell+x)^2$ incurs an additional error of at most $O_\theta(n^{-1})$, so that
\[
\int_0^\ell \frac{2x}{(R-\ell+x)^2}\,dx
=-2\log\Bigl(1-\frac{\ell}{R}\Bigr)-\frac{2\ell}{R}.
\]
All in all, denoting $\Phi(t,r):=-2\log(1+t/r)+2t/r$,
we obtain
\[ \psi_n(t)= \Phi\Bigl(t,\frac{R}{n}\Bigr) + O_\theta(n^{-1}), \]
uniformly on the grid $t\in[-1/2+2\theta,0]\cap(1/n)\mathbb Z$. Since $\Phi(t, r)$ is $C^1$ on the compact set \linebreak$[-1/2+2\theta,0] \times [1/2-\theta,1/2+\theta]$, we have $\sup_{(t, r) \in [-1/2+2\theta,0] \times [1/2-\theta,1/2+\theta]} \partial_r \Phi(t, r) < \infty $. Since
$\psi(t)=\Phi(t,1/2)$ for $t\le 0$, we obtain
\[
|\psi_n(t)-\psi(t)|\le \left | \Phi\Bigl(t,\frac{R}{n}\Bigr) - \Phi(t, 1/2) \right | + O_\theta(n^{-1})   \le  C_\theta\Big(\Big|\frac{R}{n}-\frac12\Big|+\frac1n\Big)
\]
uniformly over $t\in[-1/2+2\theta,0]\cap(1/n)\mathbb Z$ on $A_n(\theta)$. The case
$t\in[0,1/2-2\theta]\cap(1/n)\mathbb Z$ is identical, except we use the representation of $\psi_n(t)$ via
$W_{R+\ell}$ and replace $R$ by $n-R$.  \eqref{eq:bulk_unif_self} follows.
\end{proof}

Observe that, by the definition of $\psi_n(t;z),\psi(t;z)$ given in \eqref{def:psi_n} and \eqref{def:psi} respectively, 
for $t\in[-\frac{R}{n}, \frac{n-R}{n}]\cap(1/n)\mathbb Z$, we have
\begin{align*}
\left|\psi_n(t;z)-\psi(t;z)\right|&=\left|(\psi_n(t)-t\log|z|)-\frac{R}{n}\log|z|-(\psi(t)-t\log|z|)+\frac12 \log|z|\right|\\
&\le \Big|\psi_n(t)-\psi(t)\Big|+\left|\frac{R}{n}-\frac12 \right|\log|z|.
\end{align*}

Writing $I_C=(C^{-1},C)$ and
taking supremum over $|z|\in I_C$ and using $|\log|z||\le M$, Lemma~\ref{lem:R_conc} yields
\begin{equation}\label{eq: reduce term}\sup_{|z|\in I_C}\left|\frac12 \log|z|-\frac{R}{n}\log|z|\right|\le M \left|\frac R n - \frac12\right|\overset{\mathrm{p}}\to 0.\end{equation}

Hence,  Claim~\ref{cl: control psi} 
implies that for $n>$ sufficiently large, we have
\begin{equation}\label{eq: goal pointwise}
\sup_{t\in B_n(\theta)} \sup_{|z|\in I_C}
\bigl|\psi_n(t;z)-\psi(t;z)\bigr|
\stackrel{p}{\to}  0 .\end{equation}
Had the minimum of both functions in Lemma~\ref{lem:psi_n_limit} been taken on $B_n(\theta)$, we would be done. We are therefore left with showing that the minimum on this set converges to the global minimum for both functions. Denote 
\begin{align*}
\hspace{40pt}\Psi_n(t;z)&:=\psi_n(t)-t\log|z| =  \psi_n(t; z)  + \frac{R}{n} \log |z|   &&\text{for $t\in[-\frac{R}{n}, \frac{n-R}{n}]\cap(1/n)\mathbb Z$ ,\hspace{40pt}}\\
\hspace{40pt}\Psi(t;z)&:=\psi(t)-t\log|z| =  \psi(t; z) +\frac{1}{2}\log|z|,
&&\text{for $|t|<\frac12$}
.\hspace{40pt}
\end{align*}

As for
$\Psi(t;z)$, we know using our choice of $\theta$ in \eqref{eq: theta_choices} that for all $|z| \in I_C$
\begin{equation}\label{eq:min loc}
\min_{t\in (-1/2,1/2)} \Psi(t;z)=\min_{t\in B(\theta)} \Psi(t;z).
\end{equation}

On $[-1/2+2\theta,\,1/2-2\theta]$, $\psi$ is $C^1$, hence $\sup_{|t|\in B(\theta)}|\psi'(t)|<\infty$. Thus,
for $|z|\in I_C$, the function $\Psi(t;z)$ is Lipschitz on $B(\theta)$  with uniform Lipschitz constant
$K_\theta:=\sup_{|t|\le 1/2-2\theta}|\psi'(t)|+M$.
Hence, 
\begin{equation}\label{eq:min loc1}
\sup_{|z| \in I_C} \left | \min_{t\in B_n(\theta)} \Psi(t;z) -  \min_{|t| < 1/2} \Psi(t;z) \right|  = \sup_{|z| \in I_C} \left | \min_{t\in B_n(\theta)} \Psi(t;z) -  \min_{t\in B(\theta)} \Psi(t;z) \right| \le \frac{K_\theta}{n},
\end{equation}
where the equality follows from \eqref{eq:min loc}. 

To control the minimum of $\psi_n$, we 
require a lower bound on 
$\psi_n(t)$ complement $B_n(\theta)$.

\begin{claim}\label{cl: control psi2}
Denote $c(\theta)=\log\Big(\frac{1/2-\theta}{4\theta}\Big)-1$.
For $t\in\mathcal T_n\setminus B_n(\theta)$, we have
\begin{equation}\label{eq: psi_det_lb}
  \left\{\psi_n(t)\ge \log\Big(\frac{(1/2-\theta)n}{(3\theta)n+1}\Big)-1
\ge \log\Big(\frac{1/2-\theta}{4\theta}\Big)-1\right\}\supset A_n(\theta)
\end{equation}
for all sufficiently large $n$.

\end{claim}
\begin{proof}
Observe that, on $A_n(\theta)$, we have $|t|\le \tfrac12+\theta$.
Write $\ell :=-nt$ if $t\le 0$ (the case $t\ge 0$ is identical with $R$
replaced by $n-R$). Since $t\in\mathcal T_n$ we have $\ell\le R$, and on $A_n(\theta)$ also $\ell\ge \big(1/2-2\theta\big)n$ and
$R\le \big(1/2+\theta\big)n$; hence $0\le R-\ell\le 3\theta n$. Using \eqref{rep:psi_negative_l}, 
we obtain
\begin{align*}
\psi_n(t)&\ge \sum_{s=1}^{\ell}\frac{s}{(R-\ell+s)^2}
=\sum_{u=R-\ell+1}^{R}\frac{u-(R-\ell)}{u^2}\\
&=\sum_{u=R-\ell+1}^{R}\frac{1}{u}-(R-\ell)\sum_{u=R-\ell+1}^{R}\frac{1}{u^2}
\ge \log\Big(\frac{R}{R-\ell+1}\Big)-1.
\end{align*}
Since $R\ge (1/2-\theta)n$ and $R-\ell+1\le 3\theta n+1$ on $A_n(\theta)$, this yields
\[
\psi_n(t)\ge \log\Big(\frac{(1/2-\theta)n}{(3\theta)n+1}\Big)-1 \ge  \log\Big(\frac{1/2-\theta}{4\theta}\Big)-1,
\]
where the last inequality is true for $n\ge 1/\theta$ so that $(3\theta)n+1\le 4\theta n$, confirming \eqref{eq: psi_det_lb}.
\end{proof}

In contrast, on $A_n(\theta),$
by Claim~\ref{cl: control psi2}, for $t\in\mathcal T_n\setminus B_n(\theta)$ and all sufficiently large $n$, $|z| \in I_C$,
\begin{equation}\label{eq:edge_lb}
\Psi_n(t;z)\ge c(\theta)-\Big(\frac12+\theta\Big)M
\ge \inf_{|u| < 1/2}\psi(u)+\frac{M}{2}+2\varepsilon
\ge \inf_{|u| < 1/2}\Psi(u;z)+2\varepsilon,
\end{equation}

From \eqref{eq: goal pointwise} and  \eqref{eq:min loc1}, we deduce that 
\begin{equation}\label{eq: bulk_ub}
   \prob \left (   \min_{t\in B_n(\theta)}\Psi_n(t; z) \le \inf_{|t|<1/2}\Psi(t; z) +  \eps \text{ for all }  |z|\in I_C \right) \to 1.
\end{equation}
Combining \eqref{eq:edge_lb} and   \eqref{eq: bulk_ub},
we deduce that, 
\begin{equation}\label{eq:min_in_bulk_final_n}
\lim_{n\to\infty}\prob\left(\min_{t\in\mathcal T_n}\big(\psi_n(t)-t\log|z|\big)=\min_{t\in B_n(\theta)}\big(\psi_n(t)-t\log|z|\big) \  \forall |z|\in I_C \right)=1 .
\end{equation}

Putting together \eqref{eq: goal pointwise},\eqref{eq:min loc1} and \eqref{eq:min_in_bulk_final_n}, the lemma follows.
\qed
\subsubsection{Controlling bulk coefficients -- Proof of Lemma~\ref{lem:eta_bulk_small}}
\begin{proof}
Let $t_n:=(\log n)^{-1/3}$. Let $j = R+i$. We need to show that $\max_{L_n\le j\le n-L_n}|\ccW_j| > t_n$ happens with vanishing probability. 

Write $X_0:=E_0-1$ and $X_{\pm m}:=E_{\pm m}-1$ for $m\ge 1$. Then $\{X_m\}_{|m|\le n}$ are i.i.d., mean zero,
sub-exponential, and we set $K:=\|X_1\|_{\ast}<\infty$, where $\|\cdot \|_{\ast}$ is the Orlicz norm with weight function $e^{x}-1$ as defined in Proposition~\ref{prop:conc}.
Condition on $R$ and fix $j\in\{L_n,\dots,n-L_n\}$. We consider the cases $j\le R$ and $j> R$ separately.
In the case $j\le R$, letting $k:=R-j$ and re-indexing with $s:=k - m+1$ gives
\[
\ccW_{R-k} = \ccW_j=\frac{1}{n+1}\widetilde X_0+\sum_{s=1}^{k}\frac{s}{T_{j+s}} \widetilde X_s,
\]
where $\{\widetilde X_s\}$ is a subcollection of the $\{X_m\}_{|m|\le n}$. Hence $\ccW_j=\sum_{s=0}^k a_s \widetilde X_s$ with
$a_0=(n+1)^{-1}$ and $a_s=s/T_{j+s}$, so using $T_u\ge u^2/2$,
\[
\|a\|_\infty\le \max \Big(\frac{1}{n+1},  \sup_{s\ge 1}\frac{2s}{(j+s)^2} \Big) \le \max \Big( \frac{1}{n+1}, \frac{1}{2j})\le \frac{C}{L_n},
\]
and 
\[ \|a\|_2^2\le \frac{1}{n^2}+4\sum_{s\ge 1}\frac{s^2}{(j+s)^4}\le \frac{C}{j}\le \frac{C}{L_n},\]
for a universal constant $C<\infty$. The case $j\ge R$ is identical after replacing $j$ by $n-j$, and again yields
$\|a\|_\infty\le C/L_n$ and $\|a\|_2^2\le C/L_n$, uniformly in $R$ and $j\in[L_n,n-L_n]$.

By Proposition~\ref{prop:conc}, 
\[
\prob\big(|\ccW_j|\ge t_n\mid R\big)
\le 2\exp\Big(-c\min\Big(\frac{t_n^2}{K^2\|a\|_2^2},\frac{t_n}{K\|a\|_\infty}\Big)\Big)
\le 2\exp(-c' t_n^2 L_n)
=2\exp\big(-c'(\log n)^{4/3}\big),
\]
uniformly over $R$ and $j\in[L_n,n-L_n]$. A union bound gives
\[
\prob\Big(\max_{L_n\le j\le n-L_n}|\ccW_j|>t_n\ \Big|\ R\Big)
\le 2n\exp\big(-c'(\log n)^{4/3}\big)\to 0,
\]
and taking expectation over $R$ completes the proof.
\end{proof}
\subsubsection{Controlling edge coefficients -- Proof of Lemma~\ref{lem:edge}}
\begin{proof}
Let $A_n:=\{n/4\le R\le 3n/4\}$. By Lemma~4.3, $\mathbb{P}(A_n^c)\to 0$.
On $A_n$ and for all large $n$ we have $-R+L_n<0<n-R-L_n$, hence by Remark~\ref{rem:W-convex},
\[
\min_{i\in \cE_n} W_{R+i}=\min\bigl(W_{L_n},W_{n-L_n}\bigr).
\]
It suffices to show that $\mathbb{P}(W_{L_n}\ge K)\to 1$ since the bound for $W_{n-L_n}$
is identical.

Fix $R$ and work conditionally on $R$. Using Lemma~\ref{lem:rep} and reindexing with $s=R-m+1$,
\[
W_{L_n}=\frac{1}{n+1}E_0+\sum_{s=L_n+1}^{R} b_s \widetilde{E}_s,
\qquad
b_s:=\frac{s-L_n}{T_s}=\frac{2(s-L_n)}{s(s+1)},
\]
where $\{\widetilde{E}_s\}$ is a subcollection of $\{E_{-m}\}$, hence i.i.d.\ $\mathrm{exp}(1)$ and
independent of $E_0$. Since $\mathbb{E}[\widetilde{E}_s]=1$ and $\mathrm{Var}(\widetilde{E}_s)=1$,
\[
\mathbb{E}[W_{L_n}| R]=\frac{1}{n+1}+\sum_{s=L_n+1}^{R} b_s,
\qquad
\mathrm{Var}(W_{L_n} | R)=\frac{1}{(n+1)^2}+\sum_{s=L_n+1}^{R} b_s^2.
\]
On $A_n$ we have $R\ge n/4$. For $s\ge 2L_n$, $s-L_n\ge s/2$, hence $b_s\ge 1/(s+1)$ and therefore,
for all large $n$,
\[
\mathbb{E}[W_{L_n}| R]\ge \sum_{s=2L_n}^{\lfloor n/4\rfloor}\frac{1}{s+1}
\ge \log\Bigl(\frac{n}{8L_n}\Bigr)\to\infty.
\]
Also $b_s\le 2/s$, so $b_s^2\le 4/s^2$ and
\[
\mathrm{Var}(W_{L_n} | R)\le \frac{1}{n^2}+4\sum_{s=L_n+1}^{\infty}\frac{1}{s^2}
\le \frac{C_0}{L_n}\to 0,
\]
for a universal constant $C_0$. By Chebyshev's inequality,
\[
\mathbb{P}(W_{L_n}< K \mid R)\to 0,
\]
uniformly over $R\in[n/4,3n/4]$.
Taking expectation over $R$ yields $\mathbb{P}(W_{L_n}<K, A_n)\to 0$ and hence
$\mathbb{P}(W_{L_n}\ge K)\to 1$.
\end{proof}

\subsection{Proof of the lower bound}\label{sec: pf lower}
Fix $z\ne 0$. Recall $L_n:=\lfloor(\log n)^2\rfloor$ and $\cB_n:=[-R+L_n,\,n-R-L_n]\cap\mathbb Z$. Let $-R\le k_*\le n-R$ be an index (say, largest) such that
\[
\psi_n(k_*/n;z)=\min_{k\in\cB_n}\psi_n(k/n;z).
\]
To avoid the index $0$, we slightly modify the definition by setting
\[
\upsilon:=k_*+\mathbf 1_{\{k_*=0\}},
\]
so that $\upsilon \ne 0$. By Lemma~\ref{lem:psi_n_limit}, $\psi_n(k_*/n;z)\to -G(z)$ in probability. On the other hand, a direct computation shows that
\[
\big|\psi_n(0;z)-\psi_n(1/n;z)\big|
\le \frac{1}{T_{n-R}}+\frac{|\log|z||}{n}
\stackrel{p}{\to}0,
\]
by Lemma~\ref{lem:R_conc}. Combining these yields
\begin{equation}\label{eq:upsilon_to_istar}
\psi_n(\upsilon/n;z)=-G(z)+o_p(1).
\end{equation}
Thus, the lower bound \eqref{eq:logpotential_lb} is equivalent to the following.

\begin{prop}\label{prop:lb}
As $n\to\infty$,
\[
|P_n(z)|\ge e^{-o_p(n)}\,\exp\big(-n\psi_n(\upsilon/n;z)-n\ccW_{R+\upsilon}\big).
\]
\end{prop}

Indeed, by \eqref{eq:upsilon_to_istar} and Lemma~\ref{lem:eta_bulk_small}, Proposition~\ref{prop:lb} implies
\begin{align*}
\frac{1}{n}\log|P_n(z)|
&\ge -\psi_n(\upsilon/n;z)-\ccW_{R+\upsilon}-o_p(1)\\
&=G(z)-o_p(1).
\end{align*}
 The remainder of the section is devoted to the proof of Proposition~\ref{prop:lb}.

Denote
\[ F_0 = \frac{E_0  -1}{n+1}, \ \ F_{m} = \frac{E_{m} - 1}{T_{R +m+1}}, \ \ -R \le m \le -1, \quad 
F_{m} = \frac{E_{m} - 1}{T_{n-R-m+1}}, \ \ 1 \le m \le n-R. \]
Then, for $-R\le k\le n-R$,
\begin{align*}
\ccW_{R+k}&=F_0+\sum_{m=1}^{k}(k-m+1)F_m \ \text{if } \ 1\le k\le n-R,\\
\ccW_{R+k}&=F_0+\sum_{m=k}^{-1}(-k+m+1)F_m \ \  \text{if }\ -R\le k\le 0.
\end{align*}
If $\upsilon\ge 1$, define
\[
X:=F_{\upsilon}+F_{\upsilon+1}+F_{\upsilon+2},\qquad
Y:=F_{\upsilon}-F_{\upsilon+2},\qquad
Z:=F_{\upsilon}-2F_{\upsilon+1}+F_{\upsilon+2}.
\]
If $\upsilon \le -1$, define
\[
X:=F_{\upsilon}+F_{\upsilon-1}+F_{\upsilon-2},\qquad
Y:=F_{\upsilon}-F_{\upsilon-2},\qquad
Z:=F_{\upsilon}-2F_{\upsilon-1}+F_{\upsilon-2}.
\]
In either case, these relations invert to
\[
F_{\upsilon}=\frac{X}{3}+\frac{Y}{2}+\frac{Z}{6},\qquad
F_{\upsilon\pm 1}=\frac{X}{3}-\frac{Z}{3},\qquad
F_{\upsilon\pm 2}=\frac{X}{3}-\frac{Y}{2}+\frac{Z}{6}.
\]
The idea behind these linear combinations is that every variable $\ccW_{R+k}$, for $-R\le k\le n-R$ with $k\ne \upsilon$, can be expressed as a function of
\begin{align*}
\mathcal{F}_+ &=  \{ E_m : -R \le m \le n - R, m  \ne  \upsilon, \upsilon+1, \upsilon+2  \} \cup  \{  X, Y \}, \ \ \text{ if } \ \upsilon \ge 1 \\
\mathcal{F}_- &=  \{ E_m : -R \le m \le n - R, m  \ne  \upsilon, \upsilon - 1, \upsilon - 2  \} \cup  \{  X, Y \}, \ \ \text{ if } \ \upsilon \le  -1.
\end{align*}
To justify the above, we note that for $\upsilon\ge 1$ and $1\le k\le n-R, k\ne \upsilon$:
\[
\ccW_{R+k}
=F_0+\sum_{\substack{1\le m\le k\\ m\notin\{\upsilon,\upsilon+1,\upsilon+2\}}}(k-m+1)F_m
+
\begin{cases}
0, & 1\le k\le \upsilon-1,\\
(k-\upsilon)X+Y, & \upsilon+1\le k\le n-R.
\end{cases}
\]
On the other hand, for $\upsilon\le -1$ and $-R\le k\le 0, k\ne \upsilon$:
\[
\ccW_{R+k}
=F_0+\sum_{\substack{k\le m\le -1\\ m\notin\{\upsilon,\upsilon-1,\upsilon-2\}}}(-k+m+1)F_m
+
\begin{cases}
(\upsilon-k)X+Y, & -R\le k\le \upsilon-1,\\
0, & \upsilon+1\le k\le 0.
\end{cases}
\]
For $k=\upsilon,$ we have 
\[
\ccW_{R+\upsilon}
=
\begin{cases}
\displaystyle
F_0+\sum_{m=1}^{\upsilon-1}(\upsilon-m+1)F_m+\frac{X}{3}+\frac{Y}{2}+\frac{Z}{6},
& \upsilon\ge 1,\\[1.2ex]
\displaystyle
F_0+\sum_{m=\upsilon+1}^{-1}(-\upsilon+m+1)F_m+\frac{X}{3}+\frac{Y}{2}+\frac{Z}{6},
& \upsilon\le -1.
\end{cases}
\]
To prove Proposition~\ref{prop:lb}, it suffices to establish the following anticoncentration bound.
 \begin{equation}\label{eq:lb1}
\Big | e^{ -  n \ccW_{R+\upsilon}}  + \sum_{-R \le k \le n-R, k \ne \upsilon} b_{R+k}  e^{ - n \ccW_{R+ k}}  \Big| \ge e^{ - o_p(n)}  e^{ - n \ccW_{R+\upsilon}},
\end{equation}
where 
\begin{equation}\label{def:bi}
b_{R+k} = \exp \big (-n \psi_n(k/n; z) +n \psi_n(\upsilon/n; z) \big) \omega^{k - \upsilon}, \quad -R \le k \le n -R, \quad  \omega = z/|z|.
\end{equation}
From now on, we only consider the case $\upsilon\ge 1$, since the argument in the other case is identical.  

Set
\[
m:=n-R-\upsilon \in [L_n, n],\qquad 
\delta_{0}:=\frac{1}{T_{m+1}},\quad
\delta_{1}:=\frac{1}{T_{m}},\quad
\delta_{2}:=\frac{1}{T_{m-1}},
\qquad T_j=\frac{j(j+1)}{2},
\]
so that $F_{\upsilon+i}=\delta_i(E_{\upsilon+i}-1)$ for $i=0,1,2$. 
Then the joint density of $(F_{\upsilon},F_{\upsilon+1},F_{\upsilon+2})$ at $(f_{\upsilon},f_{\upsilon+1},f_{\upsilon+2})$ is
\[
\frac{e^{-3}}{\delta_{0}\delta_{1}\delta_{2}}
\exp\!\Big(-\frac{f_{\upsilon}}{\delta_{0}}-\frac{f_{\upsilon+1}}{\delta_{1}}-\frac{f_{\upsilon+2}}{\delta_{2}}\Big)
\mathbf 1_{\{f_{\upsilon}>-\delta_{0},\ f_{\upsilon+1}>-\delta_{1},\ f_{\upsilon+2}>-\delta_{2}\}}.
\]

Consequently, the joint density of $(X,Y,Z)$ at $(x,y,z)$ is
\begin{gather*}
\frac{e^{-3}}{6\delta_{0}\delta_{1}\delta_{2}}
\exp\!\Big(-\Big(T_m+\frac{1}{3}\Big)x-\Big(m+\frac{1}{2}\Big)y-\frac{z}{6}\Big)\,
\mathbf{1}_{\{x/3+y/2+z/6>-\delta_{0},\ x/3-z/3>-\delta_{1},\ x/3-y/2+z/6>-\delta_{2}\}}\\
=\frac{e^{-3}}{6\delta_{0}\delta_{1}\delta_{2}}
\exp\!\Big(-\Big(T_m+\frac{1}{3}\Big)x-\Big(m+\frac{1}{2}\Big)y-\frac{z}{6}\Big)\,
\mathbf 1_{\{\ell_-(x,y)<z<\ell_+(x,y)\}},
\end{gather*}
where the factor $6$ comes from the Jacobian
$\big|\det\frac{\partial(x,y,z)}{\partial(f_{\upsilon},f_{\upsilon+1},f_{\upsilon+2})}\big|=6$, and the constraints
$f_{\upsilon}>-\delta_{0}$, $f_{\upsilon+1}>-\delta_{1}$, $f_{\upsilon+2}>-\delta_{2}$ are equivalent to
\[
z>\max\big(-2x-3y-6\delta_{0},\ -2x+3y-6\delta_{2}\big)=:\ell_-(x,y),
\qquad
z<x+3\delta_{1}=:\ell_+(x,y).
\]
Note that the joint density is positive only if 
\begin{equation}\label{eq:ell_length}
  \ell(x, y) := \ell_{+} - \ell_{-} 
= 3 \min\big( x + y + \delta_{1} + 2\delta_{0},\ x - y + \delta_{1} + 2\delta_{2} \big) > 0.  
\end{equation}
In particular, the conditional density of $Z$ given $(X,Y)=(x,y)$ such that satisfies $\ell(x, y)> 0$
\[
p_{Z\mid X,Y}(z\mid x,y)\propto e^{-z/6}\,\mathbf 1_{\{\ell_-(x,y)<z<\ell_+(x,y)\}}.
\]
It is easy to check that 
\[
3\big( x - |y| + 3\delta_{0} \big) \ \le\ \ell(x, y) \ \le\ 3\big( x - |y| + 3\delta_{2} \big).
\]

\begin{lem}\label{lem:good_xy}
Assume $\upsilon\ge 1$ and let $s:=n-R-\upsilon$. If $\eps_n=o(s^{-2})$, then
\[
\mathbb P\big(\ell(X,Y)\ge \eps_n\big)\to 1,
\]
where $\ell(X,Y)=\ell_+(X,Y)-\ell_-(X,Y)$ is defined in \eqref{eq:ell_length}.
\end{lem}
\begin{proof}
Using \eqref{eq:ell_length} and the definitions of $X$ and $Y$, we
obtain the exact identity
\begin{align*}
\ell(X,Y)
&=3\min\bigl\{
X+Y+\delta_1+2\delta_0,\,
X-Y+\delta_1+2\delta_2
\bigr\}\\
&=3\min\bigl\{
2\delta_0E_{\upsilon}+\delta_1E_{\upsilon+1},\,
\delta_1E_{\upsilon+1}+2\delta_2E_{\upsilon+2}
\bigr\}.
\end{align*}
Conditionally on $R$, the variables
$E_{\upsilon},E_{\upsilon+1},E_{\upsilon+2}$ are independent
$\exp(1)$ random variables. Hence, using
$\prob(E<t)\le t$ for $t\ge0$,
\begin{align*}
\prob\bigl(\ell(X,Y)<\eps\mid R\bigr)
&\le
\prob\left( 2\delta_0E_{\upsilon} +\delta_1E_{\upsilon+1}<\frac{\eps}{3}
\,\middle|\,R\right) + 
\prob\left(
\delta_1E_{\upsilon+1}
+2\delta_2E_{\upsilon+2}<\frac{\eps}{3}
\,\middle|\,R\right)\\
&\le
\frac{\eps^2}{18\delta_0\delta_1}
+
\frac{\eps^2}{18\delta_1\delta_2}.
\end{align*}
Since
\[
\frac{1}{\delta_0\delta_1}
=T_{m+1}T_m\le Cm^4,
\qquad
\frac{1}{\delta_1\delta_2}
=T_mT_{m-1}\le Cm^4,
\]
the first claim follows. Finally, $m\le n$, so
\[
\prob\bigl(\ell(X,Y)<\eps_n\bigr)
\le C\eps_n^2n^4\to0
\]
whenever $\eps_n n^2\to0$.
\end{proof}
\begin{comment}

\begin{proof}
Using $\ell(x,y)\ge 3(x-|y|+3\delta_0)$, we have
\[
\mathbb P\big(\ell(X,Y)<\eps_n\big)\le 
\mathbb P\Big(X-|Y|<\frac{\eps_n}{3}\Big)
\le \mathbb P\Big(X+Y<\frac{\eps_n}{3}\Big)+\mathbb P\Big(X-Y<\frac{\eps_n}{3}\Big),
\]
Now
\[
X+Y=2F_{\upsilon}+F_{\upsilon+1}=2\delta_0(E_{\upsilon}-1)+\delta_1(E_{\upsilon+1}-1),
\]
so
\[
\Big\{X+Y<\frac{\eps_n}{3}\Big\}
=\Big\{2\delta_0E_{\upsilon}+\delta_1E_{\upsilon+1}<\frac{\eps_n}{3}+2\delta_0+\delta_1\Big\}
\subseteq
\Big\{E_{\upsilon}<\frac{\eps_n+6\delta_0+3\delta_1}{6\delta_0}\Big\}
\cap
\Big\{E_{\upsilon+1}<\frac{\eps_n+6\delta_0+3\delta_1}{3\delta_1}\Big\}.
\]
Since the exponential density is bounded above by one, we obtain 
\[
\mathbb P\Big(X+Y<\frac{\eps_n}{3}\Big)
\le \frac{(\eps_n+6\delta_0+3\delta_1)^2}{18\,\delta_0\delta_1}.
\]
Similarly, since $X-Y=F_{\upsilon+1}+2F_{\upsilon+2}= \delta_1(E_{\upsilon+1}-1)+2\delta_2(E_{\upsilon+2}-1)$,
\[
\mathbb P\Big(X-Y<\frac{\eps_n}{3}\Big)
\le \frac{(\eps_n+3\delta_1+6\delta_2)^2}{18\,\delta_1\delta_2}.
\]
Therefore,
\[
\mathbb P\big(\ell(X,Y)<\eps_n\big)
\le \frac{(\eps_n+6\delta_0+3\delta_1)^2}{18\,\delta_0\delta_1}
+\frac{(\eps_n+3\delta_1+6\delta_2)^2}{18\,\delta_1\delta_2}.
\]
Since $\delta_0,\delta_1,\delta_2\asymp s^{-2}$ and $\eps_n=o(s^{-2})$, the right-hand side tends to $0$, proving the claim.
\end{proof}
\end{comment}
Taking $\eps_n=n^{-3}$ in Lemma~\ref{lem:good_xy}, define
\[
K_n:=\bigl\{(x,y)\in\R^2:\ell(x,y)\ge n^{-3}\bigr\}.
\]
Then
\begin{equation}\label{eq:K_n}
\prob\bigl((X,Y)\notin K_n\bigr)\le Cn^{-2}\to0.
\end{equation}

Multiplying, our goal, \eqref{eq:lb1} by 
$\exp\Big(nF_0+n\sum_{m=1}^{\upsilon-1}(\upsilon-m+1)F_m + n\big(X/3 + Y/2\big)\Big)$, we can rewrite it as
\[
\big|e^{-nZ/6}-Q\big|\ge e^{-o_p(n)}e^{-n Z/6},
\]
for some $\mathcal F_+$-measurable random variable $Q$. We  then take $n^{-1} \log (\cdot)$ of both sides this reduces to
\begin{equation} \label{eq:anti_conc_abbrev}
\tfrac{1}{n}\log  \big | e^{ -  n Z/6}  - Q \big| 
\ge  - o_p(1)   - \frac{Z}{6} =  - o_p(1),
\end{equation}
where the last equality follows from
\begin{equation}\label{eq:Z_bound}
  |Z| \le 4\max_{|m|\le n}|F_m|
\le \frac{4\max_{|m|\le n} (E_m + 1)}{T_{n-R-\upsilon-1}} \le \frac{8 \max_{|m|\le n} (E_m + 1)}{L_n }  =  \frac{O_p(\log n)}{(\log n)^2}  =  O_p(1/\log n).  
\end{equation}
It remains to show \eqref{eq:anti_conc_abbrev}. Fix $\kappa>0.$
Then 
\begin{align*}
\prob\Big(\big|e^{-nZ/6}-Q\big|\le e^{-\kappa n}\Big)
&\le \prob\Big(\big|e^{-nZ/6}-Q\big|\le e^{-\kappa n},\ (X,Y)\in K_n,\ |Z|\le (\log n)^{-1/2}\Big) \\
&\quad + \prob\big((X,Y)\notin K_n\big)
+ \prob\big(|Z|>(\log n)^{-1/2}\big).
\end{align*}
The second and third terms on the right-hand side tend to zero by \eqref{eq:K_n} and \eqref{eq:Z_bound}, respectively. It therefore suffices to show that the first term converges to zero.
Hence, in order to conclude the proof it would suffice to show the following.

\begin{equation}\label{eq:final_sup_bound}
\lim_{n\to\infty}\sup_{(x,y)\in K_n,\ q\in\mathbb R}
\prob\Big(\big|e^{-nZ/6}-q\big|\le e^{-\kappa n},\ (X,Y)\in K_n,\ |Z|\le (\log n)^{-1/2}\,\Big|\,X=x,\ Y=y\Big)=0.    
\end{equation}

Fix $(x,y)\in K_n$. Conditionally on $(X,Y)=(x,y)$, the density of $Z$ is
\[
p_{Z\mid X,Y}(z\mid x,y)
=\frac{e^{-(z-\ell_-(x,y))/6}}{6\big(1-e^{-\ell(x,y)/6}\big)}
\mathbf 1_{\{\ell_-(x,y)<z<\ell_+(x,y)\}}
\le \frac{1}{6\big(1-e^{-\ell(x,y)/6}\big)}
\mathbf 1_{\{\ell_-(x,y)<z<\ell_+(x,y)\}}.
\]
Using $1-e^{-u}\ge u/2$ for $0<u\le 1$ and $\ell(x,y)\ge n^{-3}$ on $K_n$, we obtain
\begin{equation}\label{cond_density_bd}
\sup_{z}p_{Z\mid X,Y}(z\mid x,y)
=\frac{1}{6(1-e^{-\ell(x,y)/6})}
\le \frac{C}{\min\{\ell(x,y),1\}}
\le Cn^3,
\qquad (x,y)\in K_n,
\end{equation}
for an absolute constant $C$. For $q\in\mathbb R$, define
\[
S(q):=\Big\{z\in[-(\log n)^{-1/2},(\log n)^{-1/2}]\cap(\ell_-(x,y),\ell_+(x,y)):\ \big|e^{-nz/6}-q\big|\le e^{-\kappa n}\Big\}.
\]
We view $t=e^{-nz/6}$ as a change of variables. Since $z\mapsto e^{-nz/6}$ is strictly monotone on $\mathbb R$, it is bijective onto $(0,\infty)$ and
\[
z=-\frac{6}{n}\log t,\qquad dz=\frac{6}{n}\frac{dt}{t}.
\]
On $|z|\le (\log n)^{-1/2}$ we have $t=e^{-nz/6}\ge e^{-n/(6\sqrt{\log n})}$, hence $t^{-1}\le e^{n/(6\sqrt{\log n})}$. Therefore,
\[
\mathrm{Leb}(S(q))
\le \frac{6}{n}\int_{\{t:\ |t-q|\le e^{-\kappa n}\}}\frac{dt}{t}
\le \frac{6}{n}\cdot 2e^{-\kappa n}\cdot e^{n/(6\sqrt{\log n})}
\le \frac{C}{n}\exp\Big(-\kappa n+\frac{n}{6\sqrt{\log n}}\Big),
\]
uniformly in $q\in\mathbb R$.

Therefore, using the density bound \eqref{cond_density_bd}, we have
\begin{align*}
\prob\Big(\big|e^{-nZ/6}-q\big|\le e^{-\kappa n},\ |Z|\le (\log n)^{-1/2} \,\Big|\, X=x,Y=y\Big)
&\le \sup_{z}p_{Z\mid X,Y}(z\mid x,y)\cdot \mathrm{Leb}(S(q)) \\
&\le C n^{3}\cdot \frac{C}{n}\exp\Big(-\kappa n+\frac{n}{6\sqrt{\log n}}\Big)
= o(1),
\end{align*}
uniformly over $(x,y)\in K_n$ and $q\in\mathbb R$, which establishes \eqref{eq:final_sup_bound} and concludes the proof. 
\qed

%%%%%%%%%%%%%%%%%%%%%%%%%%%%%%%%%%%%%%%%%%%%%%
%% Single Appendix:                         %%
%%%%%%%%%%%%%%%%%%%%%%%%%%%%%%%%%%%%%%%%%%%%%%
%\begin{appendix} \label{sec: appendix}
%\section*{???}%% if no title is needed, leave empty 
\appendix
\section{Appendix}\label{sec: appendix}
\subsection{Standard probabilistic tools}
The following concentration inequality for sums of independent sub-exponential random variables is standard, see \cite{vershynin2020high}[Theorem 2.8.2].
\begin{prop}[Bernstein's inequality] \label{prop:conc}
    Let $X_1, X_2, \ldots, X_n$ be i.i.d.\ mean-zero sub-exponential random variables and let $a = (a_1, a_2, \ldots, a_n) \in \mathbb{R}^n$. Then for any $t \ge 0$, we have
    \[\prob \Big ( \big | \sum_{i=1}^n a_i X_i \big |  \ge t \Big)
    \le 2 \exp \Big( -c \min \Big(\tfrac{t^2}{ K^2 \| a\|_2^2}, \tfrac{t}{ K \| a\|_\infty} \Big) \Big),\]
    where $K = \| X_1\|_\ast: = \inf \{ s> 0: \E \exp( |X_1|/s) \le 2  \} < \infty$ is the Orlicz norm  with weight function $e^{x}-1$.
\end{prop}
\subsection{Criterion for convergence of the empirical root distribution}\label{sec:app root conv}
This section is devoted to the proof of Proposition~\ref{prop:bd_puncture}.
Let $r>0$.
First note that each $\varphi_n$ is subharmonic on $\mathbb C$, so that its restriction to the open set $E_r$ is also subharmonic.

Since $U$ is continuous on $\mathbb C\setminus\{0\}$, it is continuous on $E_r$.
By assumption \textup{(ii)}, for every $z\in E_r$ we have $\limsup_{n\to\infty}\varphi_n(z)\le U(z)$, and therefore its upper-semicontinuous regularization on $E_r$ given by 
\[\overline \psi(z)=\limsup_{w\to z}\lim_{n\to\infty}\varphi_n(z)=\inf_{\delta>0}\sup\{\lim_{n\to\infty}\varphi_n(w)\ :\ w\in E_r, |w-z|<\delta\},\qquad z\in E_R,\]
satisfies
\[\overline \varphi(z) \le U
\qquad\text{on }E_r.
\]
%Here, given a function $f:D\to[-\infty,\infty)$ on an open set $D\subset\mathbb C$, we denote by $f^\ast$ its upper semicontinuous regularization, defined by
%\[
%f^\ast(z)=\limsup_{w\to z} f(w)=\inf_{\delta>0}\ \sup\big\{f(w): w\in D,\ |w-z|<\delta\big\},\qquad z\in D.\]

Moreover, by \textup{(iii)} there is a countable dense set $\{z_i\}_{i\ge 1}\subset \mathbb C\setminus\{0\}$ with
$\varphi_n(z_i)\to U(z_i)$; intersecting with $E_r$ yields a countable dense set in $E_r$ with the same property.
Together with the local upper bound \textup{(i)} (which remains valid after restriction to $E_r$),
we may apply \cite[Theorem~2.4]{bloom19} on $E_r$ to conclude that
$\varphi_n \to U$ in $ L^1_{\mathrm{loc}}(E_r)$,
establishing \eqref{L1_loc_conv}.

Let $\phi\in C_c^\infty(E_r)$. Using the definition $\mu_{p_n}=\frac{1}{2\pi}\Delta\varphi_n$ and distributional
duality, we have
\[
\int_{\mathbb C}\phi\,d\mu_{p_n}
=\frac{1}{2\pi}\langle \Delta\varphi_n,\phi\rangle
=\frac{1}{2\pi}\langle \varphi_n,\Delta\phi\rangle.
\]
Since $\Delta\phi$ is smooth and compactly supported in $E_r$, the convergence
$\varphi_n\to U$ in $L^1_{\mathrm{loc}}(E_r)$ implies $\langle \varphi_n,\Delta\phi\rangle\to \langle U,\Delta\phi\rangle$.
Therefore,
\[
\int_{\mathbb C}\phi\,d\mu_{p_n}\to \frac{1}{2\pi}\langle U,\Delta\phi\rangle
=\frac{1}{2\pi}\langle \Delta U,\phi\rangle
=\int_{\mathbb C}\phi\,d\mu.
\]
Hence $\int \phi\,d\mu_{p_n}\to \int \phi\,d\mu$ for all $\phi\in C_c^\infty(E_r)$, and by standard approximation argument, the same holds for all $\phi\in C_c(E_r)$.

\medskip

Let $\phi\in C_c(\mathbb C)$ and $\varepsilon>0$.
Since $\mu(\{0\})=0$ and by assumption \eqref{eq:assump_no_accumulation_zero_origin}, choose $r>0$ so small that $\mu(D(0,2r))\le \varepsilon$ and $\mu_{p_n}(D(0,2r))\le \varepsilon$ for all $n$ sufficiently large.
Let $\ccW_r\in C_c(\mathbb C)$ satisfy $0\le \ccW_r\le 1$, $\ccW_r\equiv 1$ on $D(0,r)$, and
$\mathrm{supp}(\ccW_r)\subset D(0,2r)$.
Decompose
\[
\int \phi\,d\mu_{p_n}-\int \phi\,d\mu
=
\int \phi(1-\ccW_r)\,d(\mu_{p_n}-\mu)
+
\int \phi\ccW_r\,d(\mu_{p_n}-\mu).
\]
The function $\phi(1-\ccW_r)$ has compact support contained in $E_r$,  hence by the preceding convergence on $E_r$
\[
\int \phi(1-\ccW_r)\,d\mu_{p_n}\to \int \phi(1-\ccW_r)\,d\mu.
\]
For the second term, we use $|\phi\ccW_r|\le \|\phi\|_\infty \mathbf{1}_{D(0,2r)}$ to obtain
\[
\Big|\int \phi\ccW_r\,d(\mu_{p_n}-\mu)\Big|
\le \|\phi\|_\infty\Big(\mu_{p_n}(D(0,2r))+\mu(D(0,2r))\Big) \le 2\|\phi\|_\infty \varepsilon,
\]
where the last inequality holds for all sufficiently large $n$.
Combining the two pieces and then sending $\varepsilon\downarrow 0$ yields
\[
\int \phi\,d\mu_{p_n}\to \int \phi\,d\mu
\qquad\text{for every }\phi\in C_c(\mathbb C),
\]
which is the vague convergence $\mu_{p_n}\Rightarrow \mu$ on $\mathbb C$.

 Since each $\mu_{p_n}$ is a probability measure, the vague convergence together with $\mu(\mathbb C)=1$ implies tightness, and hence
$\mu_{p_n}\Rightarrow \mu$ weakly as probability measures on~$\mathbb C$.

%%%%%%%%%%%%%%%%%%%%%%
\subsection{Root distribution in the supercritical regime}
\begin{prop}
\label{prop:alpha_greater_than_one_convergence}
Fix $\alpha>1$, and let $\zeta_{n,1},\ldots,\zeta_{n,n}$ be the
roots of $P_{n,\alpha}$. Then, almost surely,
\[ d_L \Big( \frac1n\sum_{j=1}^n\delta_{\zeta_{n,j}}, \frac12\delta_0+\frac12\delta_\infty \Big) \to 0
\qquad\text{on }\widehat{\C}=\C\cup\{\infty\}.
\]
More precisely, almost surely,
\begin{equation}\label{eq:rescaled_conv}
   d_L \Big(  \frac1n\sum_{j=1}^n
\delta_{\,n^{1-\alpha}\log|\zeta_{n,j}|}, \nu \Big) \to 0,
\qquad
\nu(dx):=\frac{2}{(4+|x|)^2}\,dx. 
\end{equation}
\end{prop}

\begin{proof}
The proof uses coefficient estimates established in Section~\ref{sec: pf upper}.
Use the triangular-array realization $P_{n,\alpha}(z)=\sum_{k=0}^n e^{-n^\alpha W_{n,k}}z^k$
from Lemma~\ref{lem:rep}. Set $\theta(s):=\psi(s-1/2)$ and
\[ H(x):=G(e^x)=\sup_{0<s<1}\big ( sx-\theta(s)\big ), \qquad x\in\R, \]
where $\psi$ and $G$ are given in \eqref{def:psi}--\eqref{def:V}.

The bound $2n\exp(-c(\log n)^{4/3})$ in the proof of
Lemma~\ref{lem:eta_bulk_small} is summable. Thus, by
Borel--Cantelli, the centered fluctuations
$W_{n,k}-\E[W_{n,k} | R_n]$ vanish almost surely, uniformly
on every fixed bulk interval. Also, Lemma~\ref{lem:R_conc}
gives $R_n/n\to1/2$ almost surely. Claim~\ref{cl: control psi}
approximates the conditional means by $\psi((k-R_n)/n)$,
which converges uniformly to $\theta(k/n)$ on such intervals.
Consequently, almost surely,
\begin{equation}\label{eq:supercritical_bulk_profile}
\max_{\delta n\le k\le(1-\delta)n}
|W_{n,k}-\theta(k/n)|\to0,
\qquad 0<\delta<1/2.
\end{equation}
Taking a countable intersection over rational $\delta$, we fix
a realization on which these convergences hold for every
$\delta\in(0,1/2)$. Define
\[H_n(x):=\max_{0\le k\le n}
\Big ( \frac{k}{n}x-W_{n,k}\Big).
\]
Fix $M>0$. Since $\theta$ diverges at both endpoints, choose
$\delta\in(0,1/2)$ such that
\[
\inf_{s\in(0,\delta]\cup[1-\delta,1)}\theta(s)>2M+2.
\]
By \eqref{eq:supercritical_bulk_profile} and
Remark~\ref{rem:W-convex}, both maximizations are confined to
the fixed bulk interval, uniformly for $|x|\le M$ and all
sufficiently large $n$. The bulk convergence and a uniform
grid approximation therefore give
\begin{equation}\label{eq:supercritical_max_profile}
\sup_{|x|\le M}|H_n(x)-H(x)|\to0.
\end{equation}

For any degree-$n$
polynomial $q(z)=\sum_{k=0}^n c_kz^k = c_n\prod_{j=1}^n (z - \eta_j)$, 
let
\[
m(q):=\frac1{2\pi}\int_0^{2\pi}
\log|q(e^{i\tau})|\,d\tau
\]
denote its logarithmic Mahler measure. Jensen's formula gives
\[
e^{m(q)}=|c_n|\prod_{j=1}^n\max\{1,|\eta_j|\}.
\]
Expanding the factorization, we obtain
\[
|c_k|
\le |c_n|
\sum_{\substack{S\subseteq\{1,\ldots,n\}\\ |S|=n-k}}
\prod_{j\in S}|\eta_j|
\le \binom nk e^{m(q)}.
\]
Since $\binom nk\le2^n$, and the triangle inequality gives
$m(q)\le\log\sum_{k=0}^n|c_k|$, it follows that
\[
\max_k\log|c_k|-n\log2
\le m(q)
\le \max_k\log|c_k|+\log(n+1).
\]

For $x\in\R$, define
\[
q_{n,x}(z):=P_{n,\alpha}\bigl(e^{n^{\alpha-1}x}z\bigr)
=\sum_{k=0}^n c_{n,k}(x)z^k.
\]
Its coefficients satisfy
\[
\log|c_{n,k}(x)|
=n^\alpha\left(\frac{k}{n}x-W_{n,k}\right),
\]
so $\max_k\log|c_{n,k}(x)|=n^\alpha H_n(x)$.
Set
\[
J_n(x):=\frac1{n^\alpha}m(q_{n,x})
=\frac1{2\pi n^\alpha}\int_0^{2\pi}
\log\left|
P_{n,\alpha}\left(e^{n^{\alpha-1}x+i\tau}\right)
\right|\,d\tau.
\]
Applying the preceding bounds to $q_{n,x}$ and dividing by
$n^\alpha$ yields
\[
H_n(x)-\frac{\log2}{n^{\alpha-1}}
\le J_n(x)
\le H_n(x)+\frac{\log(n+1)}{n^\alpha}.
\]
Both errors vanish because $\alpha>1$. Hence
\eqref{eq:supercritical_max_profile} implies that $J_n\to H$
locally uniformly.

To recover the root distribution, put
\[
X_{n,j}:=n^{1-\alpha}\log|\zeta_{n,j}|.
\]
Since the leading coefficient of $P_{n,\alpha}$ is
$e^{-n^\alpha W_{n,n}}$, Jensen's formula gives
\begin{align*}
J_n(x)
&=-W_{n,n}
+\frac1{n^\alpha}\sum_{j=1}^n
\log\max\left\{e^{n^{\alpha-1}x},|\zeta_{n,j}|\right\}\\
&=-W_{n,n}
+\frac1n\sum_{j=1}^n\max\{x,X_{n,j}\}.
\end{align*}
Thus $J_n$ is convex, and its right derivative is the empirical
distribution function
\[
F_n(x):=J'_{n,+}(x)
=\frac1n\#\{j:X_{n,j}\le x\}.
\]
Since $J_n$ is convex and its limit $H$ is differentiable, Griffiths' lemma yields $J'_{n,+}(x)\to H'(x)$ for every $x\in\R$. By \eqref{def:V},
\[
H'(x)=
\begin{cases}
\displaystyle \frac{2}{4-x}, & x\le0,\\
\displaystyle 1-\frac{2}{4+x}, & x\ge0.
\end{cases}
\]
This is the continuous distribution function of
$\nu(dx)=2(4+|x|)^{-2}\,dx$, proving \eqref{eq:rescaled_conv}.

Finally, for every fixed $r>0$, we have
$n^{1-\alpha}\log r\to0$.
For any $\eps>0$, monotonicity bounds
$F_n(n^{1-\alpha}\log r)$ between $F_n(-\eps)$ and
$F_n(\eps)$ for all sufficiently large $n$.
Letting $n\to\infty$ and then $\eps\downarrow0$, we obtain
\[
\frac1n\#\{j:|\zeta_{n,j}|\le r\}
=F_n\bigl(n^{1-\alpha}\log r\bigr)
\to H'(0)=\frac12.
\]
Consequently, for every $0<a<b<\infty$, the proportion of roots
in $a<|z|\le b$ tends to zero, while the proportions in
$|z|\le a$ and $|z|>b$ both tend to $1/2$. This proves the first assertion.

\end{proof}

%\end{appendix}
\bibliographystyle{amsplain}
\bibliography{logconcave}
\end{document}